\documentclass[11pt,letterpaper]{amsart}
\usepackage[margin=1in]{geometry}

\usepackage{amsmath,amssymb,amsthm,mathrsfs,enumitem,xcolor,url,pifont,graphicx,tikz,array}
\usepackage[colorlinks, linktocpage, breaklinks]{hyperref}

\renewcommand{\leq}{\leqslant}
\renewcommand{\geq}{\geqslant}

\newtheoremstyle{mythm}
{.5\baselineskip}	
{.5\baselineskip}	
{}		
{}		
{\bf}	
{. }		
{ }		
{}		

\theoremstyle{mythm}
\newtheorem{theorem}{Theorem}[section]	
\newtheorem{lemma}[theorem]{Lemma}
\newtheorem{proposition}[theorem]{Proposition}
\newtheorem{corollary}[theorem]{Corollary}
\newtheorem{definition}[theorem]{Definition}
\newtheorem{example}[theorem]{Example}

\newtheorem{remark}[theorem]{Remark}
\newtheorem{question}{Question}

\title{An Adaptation of the Vietoris Topology for Ordered Compact Sets}

\author{Christopher Caruvana}
\address{School of Sciences, Indiana University Kokomo, 2300 S. Washington Street, Kokomo, IN 46902 USA}
\email{chcaru@iu.edu}
\urladdr{https://chcaru.pages.iu.edu/}

\author{Jared Holshouser}
\email{JHolshouser1321@gmail.com}
\urladdr{https://jaredholshouser.github.io/}

\date{3 Nov. 2025}



\begin{document}

\begin{abstract}
    We discuss a natural topology on powers of a space that is inspired by the Vietoris topology on compact subsets.
    We then place this topology in context with other product topologies; specifically, we compare this topology with the Tychonoff product,
    the box product, and Bell's uniform box topology.
    We identify a variety of topological properties for the specific case when the ground space is discrete.
    When the ground space is the Euclidean real line, we show that the resulting power is not Lindel\"{o}f,
    and hence, not Menger.
    This shows that, unlike the the Vietoris topology on unordered compact subsets, covering properties of the ground space need not transfer to the
    Vietoris power.
\end{abstract}

\maketitle

\section{Introduction}

It is well-known that compact sets display many of the same topological and closure properties as finite sets. An important aspect of the study of
finite sets that appears to be lacking from the study of compact sets is combinatorics. Mathematicians naturally consider repetition and order for finite sets,
but it is less commonly looked at for compact sets. When \(X\) is a topological space, \(X^{<\omega} := \bigcup_{n\in\omega} X^{n+1}\) represents the ordered
finite sets with repetition while \([X]^{<\aleph_0}\) is the unordered finite subsets of \(X\) without repetition. These are naturally topologized by the
Tychonoff product topology and the Vietoris hyperspace topology, respectively. Moreover,
it was shown in \cite{CHVietoris} that there are connections between \(X^{<\omega}\) and \([X]^{<\aleph_0}\) regarding selection games and covering properties.

In this paper, we discuss the space of ordered compact subsets and
topologize it analogously to \(X^{<\omega}\), that is, as a kind of product topology.
We call this product the Vietoris power due to its relation to the Vietoris topology on compact subsets.
Vietoris powers tend to be more complicated than their Tychonoff product relatives.
We explore many properties of these new product spaces and show that the covering property connections that hold between Tychnoff powers for finite tuples and
Vietoris hyperspaces of finite subsets do not hold with the Vietoris power of compact sets (or even of finite sets).
After submitting this paper, the authors learned from Jocelyn Bell that, in \cite{PinchedCube}, this same topology is introduced as the
\emph{pinched cube} topology, albeit with a different motivation and objective.

Even though the covering property relationships fail, Vietoris powers nevertheless introduce intriguing new topologies on powers that are
generally distinct from the box topology, the uniform box topology, and the Tychonoff product.
Vietoris powers also present one way to create new spaces from old while maintaining the density but
potentially growing the weight.

\subsection{Preliminaries}

We will use the standard definition of \(\omega\) where \(n \in \omega\)
is \(\{ m \in \omega : m \in n \}\).
Hence, given \(A \subseteq \omega\) and \(n \in \omega\),
we may write \(A \subseteq n\).
We let \([X]^{<\aleph_0}\) denote the set of all finite subsets of a set \(X\) and
\(X^{<\omega} = \bigcup_{n\in\omega} X^{n+1}\).

For a set \(X\) and \(A \subseteq X\), \(\mathbf 1_A : X \to 2\) will be used to denote the indicator (or characteristic) function
for \(A\).
For a set \(X\), we will use \(\wp(X)\) to refer to its power set and \(\# X\) to denote its cardinality.
For a function \(f : X \to Y\), we use \(\mathrm{img}(f)\) to denote the image \(\{ f(x) : x \in X \}\) of \(f\) and,
for \(A \subseteq X\), \(f|_A\) to denote the restriction of \(f\) to \(A\).
For a set \(X\) and \(f \in X^{<\omega}\), we let \(\mathrm{len}(f)\) denote the cardinality of the domain of \(f\).

We will use \(\mathcal O_X\) to denote the collection of all open covers of \(X\),
viewing \(\mathcal O\) as a \emph{topological operator}.
A topological operator is a class function defined on the class of all topological spaces.
Another topological operator that will appear here is \(\mathscr T\), the topological operator that produces all nonempty open subsets of a space \(X\).

For a set \(X\), we will let \(D(X)\) represent the set \(X\) endowed with the discrete topology.
All ordinals are assumed to have the order topology, unless otherwise stated.
For a space \(X\), we will use \(w(X)\), \(d(X)\), \(s(X)\), and \(e(X)\) to denote the \emph{weight}, \emph{density},
\emph{spread}, and \emph{extent}, respectively, of \(X\).
See \cite{Hodel} and \cite{JuhaszCardinal} for more on these, and other, cardinal functions.

A space \(X\) is \emph{crowded} (also known as \emph{dense-in-itself}) if it has no isolated points.
A space \(X\) is \emph{scattered} if every nonempty subspace \(Y\) of \(X\) contains a point isolated in \(Y\).

Throughout, for a space \(X\), we will use \(K(X)\) to denote the set of all nonempty compact subsets of \(X\) and
\(\mathbb K(X)\) to denote \(K(X)\) endowed with the Vietoris topology.
We will use the standard basis for \(\mathbb K(X)\) consisting of sets of the form
\begin{align*}
    [U_j : j\in\{1,\ldots,n\}] &=
    [U_1, \ldots, U_n]\\
    &= \left\{ K \in \mathbb K(X) : K \subseteq \bigcup_{j=1}^n U_j \wedge \bigwedge_{j=1}^n (K \cap U_j \neq \emptyset) \right\}.
\end{align*}
For a comprehensive introduction to the Vietoris topology, see \cite{MichaelSubsets}.

Any topological terms appearing undefined herein are to be understood as in \cite{Engelking}.
Any set-theoretic terms or notation undefined here are to be understood as in \cite{Kunen2011}.

\subsection{Covers and Selection Games}

We will consider two other kinds of open covers.
\begin{definition}
    For a space \(X\), an open cover \(\mathscr U\) of \(X\) is said to be
    \begin{itemize}
        \item
        an \emph{\(\omega\)-cover} of \(X\) if every finite subset of \(X\) is contained in a member of \(\mathscr U\).
        \item 
        a \emph{\(k\)-cover} of \(X\) if every compact subset of \(X\) is contained in a member of \(\mathscr U\).
    \end{itemize}
    We will let \(\Omega\) (resp. \(\mathcal K\)) be the topological operator which produces \(\Omega_X\) (resp. \(\mathcal K_X\)), the set of all
    \(\omega\)-covers (resp. \(k\)-covers) of a space \(X\).
\end{definition}

The notion of \(\omega\)-covers is commonly attributed to \cite{GerlitsNagy}, but they were already in use in \cite{McCoyOmegaCovers}
where they are referred to as \emph{open covers for finite sets}.
The notion of \(k\)-covers appears as early as \cite{McCoyKcovers} in which they are referred to as \emph{open covers for compact subsets}.

We recall the usual selection principles.
For more details on selection principles and relevant references, see \cite{ScheepersI,KocinacSelectedResults,ScheepersSelectionPrinciples,ScheepersNoteMat}.
\begin{definition}
    Let \(\mathcal A\) and \(\mathcal B\) be sets.
    Then the single- and finite-selection principles are defined, respectively, to be the properties
    \[\mathsf S_1(\mathcal A, \mathcal B) \equiv 
    \left(\forall A \in \mathcal A^\omega\right)\left(\exists B \in \prod_{n \in \omega} A_n\right)\ \{B_n : n \in \omega\} \in \mathcal B\]
    and
    \[\mathsf S_{\mathrm{fin}}(\mathcal A, \mathcal B) \equiv 
    \left(\forall A \in \mathcal A^\omega\right)\left(\exists B \in \prod_{n \in \omega} [A_n]^{<\aleph_0}\right)\ \bigcup\{B_n : n \in \omega\} \in \mathcal B.\]
    Following \cite{ScheepersNoteMat}, for a space \(X\) and topological operators \(\mathcal A\) and \(\mathcal B\),
    we write \(X \models \mathsf S_\ast(\mathcal A, \mathcal B)\), where \(\ast \in \{ 1 , \mathrm{fin} \}\),
    to mean that \(X\) satisfies the selection principle \(\mathsf S_\ast(\mathcal A_X, \mathcal B_X)\).
\end{definition}
Using this notation, recall that a space \(X\) is \emph{Menger} (resp. \emph{Rothberger}) if \(X \models \mathsf S_{\mathrm{fin}}(\mathcal O, \mathcal O)\)
(resp. \(X \models \mathsf S_1(\mathcal O, \mathcal O)\)).
We will also say that a space \(X\) is \emph{\(\omega\)-Menger} (resp. \emph{\(\omega\)-Rothberger})
if \(X \models \mathsf S_{\mathrm{fin}}(\Omega, \Omega)\) (resp. \(X \models \mathsf S_1(\Omega,\Omega)\));
\emph{\(k\)-Menger} (resp. \emph{\(k\)-Rothberger}) if \(X \models \mathsf S_{\mathrm{fin}}(\mathcal K, \mathcal K)\)
(resp. \(X \models \mathsf S_1(\mathcal K, \mathcal K)\)).
See \cite{TraditionalMenger} for a comprehensive comparison of these topological properties, along with relevant references.

Selection principles lead naturally to selection games, which, in our context, are types of topological games.
Topological games have a long history, much of which can be gathered from Telg{\'a}rsky's survey \cite{TelgarskySurvey}.
In this paper, the relevant selection games are games for two players, P1 and P2, running for \(\omega\) innings, defined explicitly below.
\begin{definition} \label{def:FiniteSelectionGame}
    Given sets \(\mathcal A\) and \(\mathcal B\), we define the \emph{finite-selection game}
    \(\mathsf{G}_{\mathrm{fin}}(\mathcal A, \mathcal B)\) for \(\mathcal A\) and \(\mathcal B\) as follows.
    In round \(n \in \omega\), P1 plays \(A_n \in \mathcal A\) and P2 responds with \(\mathscr F_n \in [A_n]^{<\aleph_0}\).
    We declare P2 the winner if \(\bigcup\{ \mathscr F_n : n \in \omega \} \in \mathcal B\).
    Otherwise, P1 wins.
\end{definition}
\begin{definition} \label{def:SingleSelectionGame}
    Given sets \(\mathcal A\) and \(\mathcal B\), we analogously define the \emph{single-selection game}
    \(\mathsf{G}_{1}(\mathcal A, \mathcal B)\) for \(\mathcal A\) and \(\mathcal B\) as follows.
    In round \(n \in \omega\), P1 plays \(A_n \in \mathcal A\) and P2 responds with \(x_n \in A_n\).
    We declare P2 the winner if \(\{x_n : n \in \omega \} \in \mathcal B\).
    Otherwise, P1 wins.
\end{definition}
The natural extensions of Definitions \ref{def:FiniteSelectionGame} and \ref{def:SingleSelectionGame} to contexts where
\(\mathcal A\) and \(\mathcal B\) are topological operators can be defined in the obvious way.

Notions of strategies naturally arise when considering games.
For definitions and corresponding notation for the strategy types relevant to the selection games considered here,
we refer the reader to \cite[Def. 3.7]{TraditionalMenger}.

Since we will refer explicitly to a particular partial ordering on selection games here, we define the partial ordering
explicitly here for the convenience of the reader.
\begin{definition}
    For two selection games \(\mathcal G\) and \(\mathcal H\), we write \(\mathcal G \leq_{\mathrm{II}} \mathcal H\)
    if each of the following hold:
    \begin{itemize}
        \item 
        \(\mathrm{II} \underset{\mathrm{mark}}{\uparrow} \mathcal G \implies \mathrm{II} \underset{\mathrm{mark}}{\uparrow} \mathcal H\),
        \item 
        \(\mathrm{II} \uparrow \mathcal G \implies \mathrm{II} \uparrow \mathcal H\),
        \item 
        \(\mathrm{I} \not\uparrow \mathcal G \implies \mathrm{I} \not\uparrow \mathcal H\), and
        \item 
        \(\mathrm{I} \underset{\mathrm{pre}}{\not\uparrow} \mathcal G \implies \mathrm{II} \underset{\mathrm{pre}}{\not\uparrow} \mathcal H\).
    \end{itemize}
    If, in addition,
    \begin{itemize}
        \item 
        \(\mathrm{I} \underset{\mathrm{cnst}}{\not\uparrow} \mathcal G \implies \mathrm{II} \underset{\mathrm{cnst}}{\not\uparrow} \mathcal H,\)
    \end{itemize}
    we write that \(\mathcal G \leq^+_{\mathrm{II}} \mathcal H\).
    In the case that \(\mathcal G \leq^+_{\mathrm{II}} \mathcal H\) and
    \(\mathcal H \leq^+_{\mathrm{II}} \mathcal G\), we write
    \(\mathcal G \leftrightarrows \mathcal H\).
\end{definition}

Recall that, for topological operators \(\mathcal A\) and \(\mathcal B\), and \(\ast \in \{1, \mathrm{fin}\}\),
\[\mathrm{I} \underset{\mathrm{pre}}{\not\uparrow} \mathsf{G}_\ast(\mathcal{A},\mathcal{B}) \equiv \mathsf{S}_\ast(\mathcal A, \mathcal B)\]
(see \cite[Prop. 15]{ClontzDualSelection} and \cite[Lemma 2.12]{CHVietoris}).
Hence, in particular, for a space \(X\), note that, if
\(\mathsf G_\ast(\mathcal A_X, \mathcal B_X) \leq_{\mathrm{II}} \mathsf G_\dagger(\mathcal C_X, \mathcal D_X)\)
where \(\ast, \dagger \in \{1, \mathrm{fin}\}\), then
\(X \models \mathsf S_\ast(\mathcal A, \mathcal B) \implies X \models \mathsf S_\dagger(\mathcal C, \mathcal D)\).

We now state one explicit relationship between a topologization of ordered finite sets with multiplicity and a topologization of the unordered
finite subsets of space in terms of selective covering properties.
\begin{theorem}[{\cite[Theorems 3.8 and 4.8]{CHVietoris}}] \label{thm:FiniteMotivation}
    For \(\ast \in \{1, \mathrm{fin} \}\) and any space \(X\),
    \[\mathsf G_\ast(\Omega_X,\Omega_X)
    \leftrightarrows \mathsf G_\ast(\mathcal O_{X^{<\omega}},\mathcal O_{X^{<\omega}})
    \leftrightarrows \mathsf G_\ast(\mathcal O_{\mathcal P_{\mathrm{fin}}(X)},\mathcal O_{\mathcal P_{\mathrm{fin}}(X)}).\]
\end{theorem}
In the context of compact sets, the single-selection equivalence fails (see \cite[Example 4.19]{CHVietoris}), but we still obtain
\begin{theorem}[{\cite[Theorem 4.15]{CHVietoris}}] \label{thm:FiniteCompactEquivalence}
    For any space \(X\),
    \[\mathsf G_{\mathrm{fin}}(\mathcal K_X, \mathcal K_X) \leftrightarrows \mathsf G_{\mathrm{fin}}(\mathcal O_{\mathbb K(X)}, \mathcal O_{\mathbb K(X)}).\]
\end{theorem}
The motivation for the present study was the search for an analog to \(X^{<\omega}\) appearing in Theorem \ref{thm:FiniteMotivation}
in the compact setting of Theorem \ref{thm:FiniteCompactEquivalence}.
Though we have not been successful in identifying such an analog, we have identified a natural topologization of powers
which can have interesting topological properties, which we will elaborate on below.

\subsection{The Vietoris Power}

\begin{definition} \label{def:VietorisPower}
    Let \(X\) be a space and \(\kappa\) be a cardinal.
    For \(U \subseteq X\), \(\Lambda \in [\kappa]^{<\aleph_0}\), and \(V : \Lambda \to \wp(X)\), define
    \[[U ; \Lambda, V ] = \left\{ f \in X^\kappa : \left[\forall \alpha \in \Lambda\ (f(\alpha) \in V_\alpha)\right] \wedge \mathrm{img}(f) \subseteq U \right\}.\]    
    We will see in Proposition \ref{prop:BasisProposition} that these sets form a basis for a topology
    when the \(U\) and \(V\) range over open subsets of \(X\).
    The notation for \(X^\kappa\) with this topology is \(\mathsf{V}(X^\kappa)\),
    and we refer to this space as the \emph{Vietoris power of order \(\kappa\)}.
\end{definition}
Note that
\[[U; \emptyset, \emptyset] = U^\kappa\]
and \[[X; \Lambda, V] = \bigcap_{\alpha \in \Lambda} \pi_{\alpha}^{-1}(V_\alpha).\]
We refer to the former as \emph{tubes}.
\begin{proposition} \label{prop:BasisProposition}
    For a space \(X\) and a cardinal \(\kappa\), the sets of the form \([U; \Lambda, V]\),
    where \(U \in \mathscr T_X\), \(\Lambda \in [\kappa]^{<\aleph_0}\), and \(V : \Lambda \to \mathscr T_X\),
    constitute a basis for a topology on \(X^\kappa\).
\end{proposition}
\begin{proof}
    Suppose \(f \in [U_1; \Lambda_1, V_1] \cap [ U_2; \Lambda_2, V_2]\).
    Let \(U = U_1 \cap U_2\) and \(\Lambda = \Lambda_1 \cup \Lambda_2\).
    Note that \(\Lambda \in [\kappa]^{<\aleph_0}\).
    Then define \(V : \Lambda \to \mathscr T_X\) by
    \[V_{\alpha} = \begin{cases}
        V_1(\alpha), & \alpha \in \Lambda_1 \setminus \Lambda_2,\\
        V_2(\alpha), & \alpha \in \Lambda_2 \setminus \Lambda_1,\\
        V_1(\alpha) \cap V_2(\alpha), & \alpha \in \Lambda_1 \cap \Lambda_2.
    \end{cases}
    \]
    We claim that
    \[f \in [U; \Lambda, V] \subseteq [U_1; \Lambda_1, V_1] \cap [ U_2; \Lambda_2, V_2].\]
    
    Clearly, \(\mathrm{img}(f) \subseteq U_1 \cap U_2 = U\).
    Also, the fact that \(f(\alpha) \in V_{\alpha}\) for each \(\alpha \in \Lambda\) is evident.
    
    So we finish by showing that 
    \[[U; \Lambda, V] \subseteq [U_1; \Lambda_1, V_1] \cap [ U_2; \Lambda_2, V_2].\]
    Suppose \(g \in [U; \Lambda, V]\).
    Immediately, it is seen that \(\mathrm{img}(g) \subseteq U_j\) for \(j = 1,2\).
    It is also clear that, for any \(\alpha \in \Lambda_j\), \(j = 1,2\), \(g(\alpha) \in V_j(\alpha)\).
    Hence,
    \[g \in [U_1; \Lambda_1, V_1] \cap [ U_2; \Lambda_2, V_2].\]
    This completes the proof.
\end{proof}
\begin{remark}
    If \(\kappa\) is a finite cardinal, then the box topology, Vietoris power, and Tychonoff product on \(X^\kappa\) all coincide.
\end{remark}
\begin{remark}
    Note that the natural mapping \(x \mapsto \vec{x}\), \(X \to \mathsf V(X^\kappa)\), where
    \(\vec{x}\) represents the constant function \(\kappa \to X\) taking value \(x\),
    is a continuous injection.
    If \(X\) is Hausdorff, it's evident that the image of \(X\) under this map is closed.
\end{remark}
\begin{proposition} \label{prop:BasicProductOfClosed}
    Let \(X\) be a space and \(\kappa\) be a cardinal.
    Suppose \(E \subseteq X\) is closed, \(\Lambda \in [\kappa]^{<\aleph_0}\), and \(F : \Lambda \to \wp(X)\)
    is such that \(F_\alpha\) is closed for each \(\alpha \in \Lambda\).
    Then \([E; \Lambda, F]\) is closed in \(\mathsf V(X^\kappa)\).
\end{proposition}
\begin{proof}
    Suppose \(f : \kappa \to X\) is such that \(f \not\in [E; \Lambda, F]\).
    In the case that \(\mathrm{img}(f) \not\subseteq E\), let \(\beta < \kappa\) be such that
    \(f(\beta) \not\in E\).
    Note then that
    \[f \in [X; \{\beta\}, \{\langle \beta, X \setminus E \rangle \}] \subseteq \mathsf V(X^\kappa) \setminus [E; \Lambda, F].\]
    Otherwise, we can find \(\alpha \in \Lambda\) with \(f(\alpha) \not\in F_\alpha\).
    Then,
    \[f \in [X; \{\alpha\}, \{ \langle \alpha, X \setminus F_\alpha \rangle \}] \subseteq \mathsf V(X^\kappa) \setminus [E;\Lambda, F].\]
    Hence, \(\mathsf V(X^\kappa) \setminus [E;\Lambda, F]\) is open, finishing the proof.
\end{proof}
\begin{remark} \label{rmk:DiscreteSpaceBasis}
    For a discrete space \(X\), we can use
    \[[s,A] = \{ f \in X^\omega : f|_{\mathrm{len}(s)} = s \wedge \mathrm{img}(f) \subseteq A \},\]
    where \(s \in X^{<\omega}\) and \(A \subseteq X\),
    as a basis for \(\mathsf V(X^\omega)\).
    Note that each \([s,A]\) is clopen by Proposition \ref{prop:BasicProductOfClosed}.
    When convenient, we will use \([\![ s ]\!] = [s, \mathrm{img}(s)]\)
    for \(s \in X^{<\omega}\).
\end{remark}

In Section \ref{section:ExamplesAndResults}, we will compare the Vietoris power topology to the Tychonoff product,
the box product, and the uniform box topology, which was introduced in \cite{BellUniformBox}.

To capture the ordered compact sets, we can look at particular subspaces of \(\mathsf V(X^\kappa)\) created by limiting the range of the functions.
\begin{definition}
    For a space \(X\) and a cardinal \(\kappa\), let \[\mathbb K(X,\mathrm{ord},\kappa) = \{ f \in X^\kappa : \mathrm{img}(f) \in K(X) \}\]
    and endow it with the topology it inherits as a subspace of \(\mathsf V(X^\kappa)\).
    Then let \[\mathbb K(X,\mathrm{ord}) = \mathbb K(X,\mathrm{ord}, \sup\{\#K : K \in K(X)\} + \omega).\] 
\end{definition}
For a point of comparison, we will also provide notation for the subspace of \(\mathsf V(X^\kappa)\) consisting of functions with finite range.
\begin{definition}
    For a space \(X\) and a cardinal \(\kappa\), let
    \[\mathbb F(X,\mathrm{ord},\kappa) = \left\{ f \in X^\kappa : \mathrm{img}(f) \in [X]^{<\aleph_0} \right\}\]
    endowed with the topology it inherits as a subspace of \(\mathsf V(X^\kappa)\).
    Then let \(\mathbb F(X,\mathrm{ord}) = \mathbb F(X, \mathrm{ord}, \omega)\).
\end{definition}
\begin{remark} \label{rmk:FiniteDensity}
    It is evident that \(\mathbb F(X,\mathrm{ord},\kappa) \subseteq \mathbb K(X,\mathrm{ord}, \kappa)\) and that
    \(\mathbb F(X, \mathrm{ord}, \kappa)\) is dense in \(\mathsf V(X^\kappa)\).
\end{remark}

Note that \(\bigcup_{n\in\omega} \mathsf V(X^n)\) is equal to \(X^{<\omega}\), topologized as a disjoint union of \(X^n\) (the Tychonoff product).
Thus, if we want to study a new space, we are forced to look at \(\mathbb F(X, \mathrm{ord}, \omega)\).

Also, when \(X\) is anticompact, \(\mathbb F(X,\mathrm{ord}) = \mathbb K(X,\mathrm{ord})\) (recall that a space is said to be \emph{anticompact},
following \cite{Bankston}, if every compact subset is finite).

Before we continue elaborating on the Vietoris power, we offer some additional motivation for the choice in
Definition \ref{def:VietorisPower} over using a disjoint union, despite the connection in Theorem \ref{thm:FiniteMotivation} with \(X^{<\omega}\).

Note that \(X := \aleph_{\omega_1 + \omega}\) is hemicompact, and thus \(k\)-Rothberger.
Indeed, \(\{ \aleph_{\omega_1 + n} + 1 : n \in \omega \}\) is a countable family of compact subsets of
\(X\) and every compact subset of \(X\) must be contained in some \(\aleph_{\omega_1+n}\), for \(n \in \omega\).
Now let \(K_\kappa(X) = \{ f \in X^\kappa : \mathrm{img}(f) \in K(X) \}\) and note that \(X\)
has a compact subset of each cardinality \(\kappa < \aleph_{\omega_1 + \omega}\).
It follows that the disjoint union
\[\bigsqcup\{ K_\kappa(X) : \kappa \in \mathrm{CARD} \cap \aleph_{\omega_1+\omega} \}\]
is not Lindel\"{o}f as it is an uncountable disjoint union.

Even if we were to define \(K_\kappa^\ast(X) = \{ f \in K_\kappa(X) : \# \mathrm{img}(f) = \kappa \}\), the disjoint union
\[\bigsqcup\{ K^\ast_\kappa(X) : \kappa \in \mathrm{CARD} \cap \aleph_{\omega_1+\omega} \}\]
is also not Lindel\"{o}f as it is an uncountable disjoint union.
Note additionally that the cardinality stipulation in \(K_\kappa^\ast(X)\) is not mirrored in the \(X^{<\omega}\) context.

As a final comment on this topic, we observe that the disjoint union context introduces potentially undesirable
redundancy.
For example, note that, for every cardinal \(\kappa < \mathfrak c\), every element of \(K_\kappa(\mathbb R)\)
is reflected also in \(K_{\mathfrak c}(\mathbb R)\).

\section{Examples and Results}
\label{section:ExamplesAndResults}

\subsection{Properties of Vietoris Powers}

We start by showing that the natural analog to the Tychonoff Theorem on products for Vietoris powers does not hold.
\begin{example} \label{example:Cantor}
    \(\mathsf{V}(X^\kappa)\) need not be compact even when \(X\) is compact.
\end{example}
\begin{proof}
    We show that \(\mathsf V(2^\omega)\) is not compact.
    Let \(U = \{0\}^\omega\) and, for each \(n \in \omega\), \(V_n = \{ b \in 2^\omega : b_n = 1 \}\).
    Note that \(U\) is open and each \(V_n\) is open.
    Moreover, \(\mathscr U := \{U\} \cup \{V_n : n\in \omega\}\) forms a cover of \(\mathsf{V}(2^\omega)\).
    However, \(\mathscr U\) has no finite subcover.
    Indeed, consider any \(F \in [\omega]^{<\aleph_0}\) and the collection
    \(\{U\} \cup \{V_n : n \in F \}\).
    Let \(m = 1 + \max F\) and define \(b : \omega \to 2\) by
    \[b_n = \begin{cases} 0, & n \neq m,\\ 1, & n = m.\end{cases}\]
    Observe that \(b \not\in U \cup \bigcup_{n \in F} V_n\).
\end{proof}
\begin{remark}
    Note that this also shows that, even if we restrict to \(\mathbb K(X, \mathrm{ord}, \kappa)\) for compact \(X\),
    we can't guarantee compactness since \(\mathsf V(2^\omega) = \mathbb K(2, \mathrm{ord}, \omega) = \mathbb K(2, \mathrm{ord})\).
    Moreover, this also shows that, in general, given some compact subset \(K_0\) of a set \(X\),
    even though \(\{ K \in \mathbb K(X) : K \subseteq K_0 \}\) is compact (see \cite{MichaelSubsets}), the ordered analog,
    \(\{ f \in \mathbb K(X, \mathrm{ord}) : \mathrm{img}(f) \subseteq K_0 \}\) need not be compact.
\end{remark}
\begin{proposition}
    Consider a space \(X\) and a cardinal \(\kappa\).
    Let \(\mathscr F \subseteq \mathbb K(X, \mathrm{ord},\kappa)\) be so that
    \(\bigcup \{ \mathrm{img}(f) : f \in \mathscr F\}\) is not compact. Then \(\mathscr F\) is not compact.
\end{proposition}
\begin{proof}
    As shown below in Proposition \ref{prop:ImageMappingQuotient}, the mapping
    \(f \mapsto \mathrm{img}(f)\), \(\mathbb K(X,\mathrm{ord},\kappa) \to \mathbb K(X)\), is continuous.
    Hence, if \(\mathscr F \subseteq \mathbb K(X, \mathrm{ord},\kappa)\) is compact,
    then \(\{ \mathrm{img}(f) : f \in \mathscr F \} \subseteq \mathbb K(X)\) is compact.
    It follows (see \cite{MichaelSubsets}) that \(\bigcup \{ \mathrm{img}(f) : f \in \mathscr F \} \subseteq X\) is compact.
\end{proof}
\begin{corollary} \label{cor:NotCompactCondition}
    Let \(\kappa\) be a cardinal.
    \begin{itemize}
        \item If \(\mathscr F \subseteq \mathbb K(D(\kappa),\mathrm{ord})\) is so that \(\bigcup_{f \in \mathscr F} \mathrm{img}(f)\) is infinite,
        then \(\mathscr F\) is not compact.
        \item If \(\mathscr F \subseteq \mathbb K(\kappa,\mathrm{ord})\) is so that \(\bigcup_{f \in \mathscr F} \mathrm{img}(f)\) is unbounded,
        then \(\mathscr F\) is not compact.
        \item If \(\mathscr F \subseteq \mathbb K(\mathbb R,\mathrm{ord})\) is so that \(\bigcup_{f \in \mathscr F} \mathrm{img}(f)\) is unbounded,
        then \(\mathscr F\) is not compact.
    \end{itemize}
\end{corollary}
\begin{example} \label{ex:CantorNonRothberger}
    The space \(\mathsf V(2^\omega)\) is not Rothberger since it supports a nonatomic Borel probability measure.
    Thus, any space \(X\) which contains \(\mathsf V(2^\omega)\) as a closed subspace also fails to be Rothberger.
\end{example}
\begin{proposition} \label{prop:LargeDiscreteNotLindelof}
    For any uncountable discrete space \(X\),
    \(\mathbb K(X,\mathrm{ord})\) is not Lindel\"{o}f.
\end{proposition}
\begin{proof}
    Consider
    \[\mathscr U := \left\{ [F;\emptyset,\emptyset] : F \in [X]^{<\aleph_0} \right\}.\]
    This is evidently an open cover of \(\mathbb K(X,\mathrm{ord})\) without a countable subcover.
\end{proof}
\begin{remark}
    Given the basis structure for \(\mathsf V(X^\kappa)\), it is evident that \(w(X^\kappa) \leq w(\mathsf V(X^\kappa))\).
    We will show in Theorem \ref{thm:BairePropertyList} that this inequality can be strict.
    It is also immediate that \(d(X^\kappa) \leq d(\mathsf V(X^\kappa))\) since any dense subset of \(\mathsf V(X^\kappa)\)
    is dense in \(X^\kappa\).
\end{remark}

As a modification to the standard argument that \(d\left(D(\kappa)^{\mathrm{exp}\,\kappa}\right) \leq \kappa\), where \(\kappa\) is an infinite cardinal
and \(\mathrm{exp}\,\kappa = 2^\kappa\), we obtain
\begin{lemma} \label{lem:InitialDensity}
    For any infinite cardinal \(\kappa\), \(d\left(\mathsf V\left(D(\kappa)^{\mathrm{exp}\,\kappa}\right)\right) \leq \kappa\).
    In particular, there is a dense subset of \(\mathsf V\left(D(\kappa)^{\mathrm{exp}\,\kappa}\right)\) of cardinality
    \(\kappa\) consisting of functions with finite range.
    Such a dense set is thus contained in \(\mathbb F(D(\kappa), \mathrm{ord}, 2^\kappa) = \mathbb K(D(\kappa), \mathrm{ord}, 2^\kappa)\).
\end{lemma}
\begin{proof}
    For \(E \in [\kappa]^{<\aleph_0}\) and \(f : E \to 2\), let \[B(f,E) = \{ g \in 2^\kappa : g|_E = f \}.\]
    Note then that \(\mathcal B := \left\{ B(f,E) : E \in [\kappa]^{<\aleph_0}, f \in 2^E \right\}\) is a basis for \(2^\kappa\)
    of size \(\kappa\).
    For \(n \in \omega\), let \[\mathcal B_n^d = \left\{ \langle U_1, \ldots, U_n \rangle \in \mathcal B^n : (\forall 1 \leq j < k \leq n)\ U_j \cap U_k
    = \emptyset \right\}.\]
    For each \(\mathcal E = \langle U_1, \ldots, U_n , \alpha_1, \ldots, \alpha_n , \beta \rangle\)
    where \(\langle U_1, \ldots, U_n \rangle \in \mathcal B_n^d\) and \(\langle \alpha_1, \ldots, \alpha_n , \beta \rangle \in \kappa^{n+1}\),
    let \(\mathcal E^c = 2^\kappa \setminus \left( \bigcup_{j=1}^n U_j \right)\) and
    \[G_\mathcal E = \beta \cdot \mathbf 1_{\mathcal E^c} + \sum_{j=1}^n \alpha_j \cdot \mathbf 1_{U_j}.\]
    Note that \(\mathcal G := \left\{ G_\mathcal E : n \in \omega, \mathcal E \in \mathcal B_n^d \times \kappa^{n+1} \right\}\)
    is of size \(\kappa\) and consists of functions with finite range.
    
    We now only need show that \(\mathcal G\) is dense in \(\mathsf V\left(D(\kappa)^{\mathrm{exp}\,\kappa} \right)\). 
    First, in a similar way as above, define
    \[B^\ast(f,E,A) = \{ g \in \kappa^{\mathrm{exp}\,\kappa} : g|_E = f, \mathrm{img}(g) \subseteq A\}\]
    for \(E \in [2^\kappa]^{<\aleph_0}\), \(f : E \to \kappa\), and \(A \subseteq \kappa\); note that
    \[\mathcal B^\ast := \{ B^\ast(f,E,A) : E \in [2^\kappa]^{<\aleph_0}, f \in \kappa^E , A \subseteq \kappa \}\] forms a basis for
    \(\mathsf V\left(D(\kappa)^{\mathrm{exp}\,\kappa}\right)\).
    We complete the proof by showing that each nonempty member of \(\mathcal B^\ast\) intersects \(\mathcal G\).
    So consider \(E = \{ g_1 , \ldots, g_n \} \subseteq 2^\kappa\), \(f : E \to \kappa\), and \(A \subseteq \kappa\)
    such that \(B^\ast(f,E,A) \neq \emptyset\).
    Since \(2^\kappa\) is a Hausdorff space, we can find \(\langle U_1, \ldots , U_n \rangle \in \mathcal B_n^d\)
    such that \(g_j \in U_j\) for each \(1 \leq j \leq n\).
    Then let \(\beta = \min A\) and note that \(G_{\mathcal E} \in B^\ast(f,E,A)\) where
    \(\mathcal E = \langle U_1, \ldots, U_n , f(g_1), \ldots, f(g_n), \beta \rangle\).
\end{proof}
Lemma \ref{lem:InitialDensity} allows us to prove a version of the Hewitt-Marczewski-Pondiczery Theorem
for Vietoris products.
\begin{theorem} \label{thm:Density}
    Suppose \(\kappa \leq 2^{d(X)}\) for a space \(X\).
    Then \(d(\mathbb F(X,\mathrm{ord},\kappa)) \leq d(X)\).
    Consequently, \(d(\mathbb K(X,\mathrm{ord},\kappa)) \leq d(X)\) and
    \(d(\mathsf V(X^\kappa)) \leq d(X)\).
\end{theorem}
\begin{proof}
    Let \(\phi : d(X) \to X\) be an injection with a dense image and let \(\iota : \kappa \to 2^{d(X)}\) be an injection.
    It is evident that \(\mathrm{img}(\phi)^\kappa\) is dense in \(\mathsf V(X^\kappa)\).
    Then define \[\Phi : \mathsf V\left(D(d(X))^{\mathrm{exp}\,d(X)}\right) \to \mathsf V(X^\kappa)\] by \(\Phi(f) = \phi \circ f \circ \iota\).
    Note that \(\Phi\) maps onto \(\mathrm{img}(\phi)^\kappa\)
    and that, if \(f \in D(d(X))^{\mathrm{exp}\,d(X)}\) has finite range,
    then \(\Phi(f)\) also has finite range.
    So, to finish the proof, we will establish that \(\Phi\) is continuous.
    Toward this end, let \(f \in d(X)^{\mathrm{exp}\,d(X)}\) be such that
    \(\Phi(f) \in [U; \Lambda, V]\) for some \(U \in \mathscr T_X\), \(\Lambda \in [\kappa]^{<\aleph_0}\), and \(V : \Lambda \to \mathscr T_X\).
    Let \(\Lambda^\ast = \iota[\Lambda] \in [2^{d(X)}]^{<\aleph_0}\) and define \(W : \Lambda^\ast \to \wp(d(X))\)
    by the rule \(W_\alpha = \phi^{-1}\left(V_{\iota^{-1}(\alpha)}\right)\).
    Note then that \(O := \left[ \phi^{-1}(U); \Lambda^\ast , W \right]\) is a neighborhood of \(f\)
    in \(\mathsf V\left(D(d(X))^{\mathrm{exp}\,d(X)}\right)\) with \(\Phi[O] \subseteq [U;\Lambda,V]\).
    Hence, \(\Phi\) is continuous.

    By Lemma \ref{lem:InitialDensity}, we can find a dense subset \(D\) of \(\mathsf V\left(D(d(X))^{\mathrm{exp}\,d(X)}\right)\)
    of cardinality \(\kappa\) consisting of functions with finite range.
    Since \(\mathrm{img}(\Phi) = \mathrm{img}(\phi)^\kappa\) is dense in \(\mathsf V(X^\kappa)\), \(\Phi[D] \subseteq \mathbb F(X,\mathrm{ord}, \kappa)\)
    is dense in \(\mathsf V(X^\kappa)\).
    The conclusion of the theorem thus is obtained.
\end{proof}
\begin{proposition} \label{prop:NonDiscrete}
    For any space \(X\) with at least two points and any infinite cardinal \(\kappa\), \(\mathsf V(X^\kappa)\)
    has non-isolated points.
    Moreover, the set of non-constant functions of \(\mathsf V(X^\kappa)\) is a crowded subspace.
\end{proposition}
\begin{proof}
    We first show that \(\mathsf V(X^\kappa)\) has non-isolated points.
    Let \(p, q \in X\) be distinct points.
    Define \(f : \kappa \to X\) by the rule
    \[f(\alpha) = \begin{cases} p, & \alpha = 0,\\
        q, & \alpha > 0.
    \end{cases}\]
    Note that any neighborhood of \(f\) in \(\mathsf V(X^\kappa)\) contains a function \(\kappa \to X\)
    which takes value \(p\) for all but finitely inputs.
    Hence, \(\{ f \}\) is not an open subset of \(\mathsf V(X^\kappa)\).

    Let \(Y\) consist of all non-constant functions \(\kappa \to X\).
    Let \(f \in Y\) be arbitrary and \([U; \Lambda, V]\) be an arbitrary basic neighborhood of \(f\).
    Then let \(\alpha_1, \alpha_2 \in \kappa \setminus \Lambda\) be such that \(\alpha_1 \neq \alpha_2\)
    and choose \(p \in \mathrm{img}(f) \setminus \{ f(\alpha_1) \}\) and
    \(q \in \mathrm{img}(f) \setminus \{ p \}\).
    Define \(g : \kappa \to X\) by
    \[g(\alpha) = \begin{cases} f(\alpha), & \alpha \in \kappa \setminus \{ \alpha_1, \alpha_2 \},\\ p, & \alpha = \alpha_1,\\ q, & \alpha = \alpha_2.\end{cases}\]
    Note that \(g \neq f\) and that \(g \in Y \cap [U; \Lambda, V]\).
    Hence, \(f\) is not isolated in \(Y\).
    Since \(f \in Y\) was arbitrary, \(Y\) is crowded.
\end{proof}
\begin{proposition} \label{prop:BasicComparison}
    The topology of \(\mathsf V(X^\kappa)\) is finer than the Tychonoff topology on \(X^\kappa\),
    coarser than the box topology on \(X^\kappa\), and it need not be homeomorphic to either.
\end{proposition}
\begin{proof}
    It is immediate from the definitions that the Tychonoff topology on \(X^\kappa\) is coarser than the topology
    on \(\mathsf V(X^\kappa)\), and that the topology on \(\mathsf V(X^\kappa)\) is coarser than the box topology
    on \(X^\kappa\).
    To see that, in general, the topology on \(\mathsf V(X^\kappa)\) need not coincide with either the Tychonoff
    nor the box topology, note that the Tychonoff topology on \(2^\omega\) is compact
    and the box topology on \(2^\omega\) is discrete.
    However, by Example \ref{example:Cantor}, \(\mathsf V(2^\omega)\) is not compact, so it cannot be homeomorphic
    to the Tychonoff topology on \(2^\omega\).
    On the other hand, Proposition \ref{prop:NonDiscrete} guarantees that \(\mathsf V(2^\omega)\) is not discrete.
    Consequently, \(\mathsf V(2^\omega)\) cannot be homeomorphic to the box topology on \(2^\omega\).
\end{proof}

We will now turn our attention to uniform spaces (see \cite[{\S}8.1]{Engelking}).
\begin{definition} [{\cite{BellUniformBox}}]
    Given a uniform space \((X, \mathbb D)\), a point \(x \in X\), and \(E \in \mathbb D\),
    we use the notation
    \[E[x] = \{ y \in X : \langle x,y \rangle \in E \}.\]
    Then, for a cardinal \(\kappa\) and \(E \in \mathbb D\), we use the notation
    \[\bar{E} = \left\{ \langle f,g \rangle \in (X^\kappa)^2 : (\forall \alpha \in \kappa)\ \langle f(\alpha),g(\alpha) \rangle \in E \right\}.\]
    We then let \(\bar{\mathbb D}\) be the uniformity on \(X^\kappa\) generated by \(\left\{ \bar{E} : E \in \mathbb D \right\}\).
    Given a uniform space \((X, \mathbb D)\), the \emph{uniform box topology} on \(X^\kappa\), denoted by
    \(\left(X^\kappa , \bar{\mathbb D} \right)\), is the topology on \(X^\kappa\) generated by the uniformity
    \(\bar{\mathbb D}\).
\end{definition}
The uniform box topology and the Vietoris power are generally incomparable topologies.
\begin{example} \label{example:UniformBox}
    Consider the standard uniformity \(\mathbb D\) on \(\mathbb R\) generated by the uniformity base \(\{ E_\varepsilon : \varepsilon > 0\}\) where
    \[E_\varepsilon := \{ \langle x, y \rangle \in \mathbb R^2 : |x-y| < \varepsilon \}.\]
    Then the topologies on \(\mathbb R^\omega\) corresponding to \(\left(\mathbb R^\omega, \bar{\mathbb D}\right)\) and \(\mathsf V(\mathbb R^\omega)\)
    are incomparable.
\end{example}
\begin{proof}
    First, consider the tube \((0,1)^\omega\) in \(\mathsf V(\mathbb R^\omega)\) and define
    \(f : \omega \to \mathbb R\) by the rule \(f(n) = 2^{-n-1}\).
    Note that \(f \in (0,1)^\omega\).
    Then note that every neighborhood of \(f\) in \(\left( \mathbb R^\omega, \bar{\mathbb D} \right)\) fails to be
    contained in \((0,1)^\omega\) as the neighborhood must allow for functions with negative outputs.
    Hence, \((0,1)^\omega\) is not open in \(\left( \mathbb R^\omega, \bar{\mathbb D} \right)\).

    Now consider the entourage \(E_1\) and \(g : \omega \to \mathbb R\) defined by \(g(n) = n\).
    For each \(m \in \omega\), define \(h_m : \omega \to \mathbb R\) by the rule
    \[h_m(n) = \begin{cases} n, & n \leq m, \\ 0, & n > m. \end{cases}\]
    To see that \(\bar{E_1}[g]\) is not open in \(\mathsf V(\mathbb R^\omega)\), consider the fact that any neighborhood
    of \(g\) in \(\mathsf V(\mathbb R^\omega)\) contains \(h_m\) for some \(m \in \omega\).
    Note also that \(h_m \not\in \bar{E_1}[g]\).
\end{proof}

For a uniform space \((X,\mathbb D)\) and a cardinal \(\kappa\), Figure \ref{fig:Comparisons}
details the relationships between the various product topologies discussed here, where \(\mathsf b(X^\kappa)\)
denotes the set \(X^\kappa\) endowed with the box topology.
The \(\to\) arrow means that that the former topology is finer than the latter;
the \(\not\to\) arrow means that the \(\to\) relation fails to hold, in general.
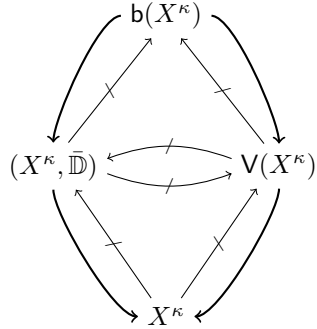
\begin{figure}[h]
    \begin{tikzpicture}
        \node (box) at (0,2) {\(\mathsf b(X^\kappa)\)};
        \node (Vietoris) at (1.5,0) {\(\mathsf V(X^\kappa)\)};
        \node (uniformBox) at (-1.5,0) {\((X^\kappa, \bar{\mathbb D})\)};
        \node (product) at (0,-2) {\(X^\kappa\)};
    
        \draw [thick,->] (box.west) to[out=180,in=90,looseness=0.4] (uniformBox.north);
        \draw [thick,->] (box.east) to[out=0,in=90,looseness=0.4] (Vietoris.north);
    
        \draw [thick,->] (uniformBox.south) to[out=270,in=180,looseness=0.4] (product.west);
        \draw [thick,->] (Vietoris.south) to[out=270,in=0,looseness=0.4] (product.east);
    
        \draw[->] ([yshift=1mm] Vietoris.west) to[bend right=20] node[midway,sloped] {\tiny$/$} ([yshift=1mm] uniformBox.east);
        \draw[->] ([yshift=-1mm] uniformBox.east) to[bend right=20] node[midway,sloped] {\tiny$/$} ([yshift=-1mm] Vietoris.west);

        \draw[->] ([xshift=-2mm] Vietoris.north) to node[midway,sloped] {\tiny\(/\)} ([xshift=2mm] box.south);
        \draw[->] ([xshift=2mm] uniformBox.north) to node[midway,sloped] {\tiny\(/\)} ([xshift=-2mm] box.south);

        \draw[<-] ([xshift=-3mm] Vietoris.south) to node[midway,sloped] {\tiny\(/\)} ([xshift=2mm] product.north);
        \draw[<-] ([xshift=3mm] uniformBox.south) to node[midway,sloped] {\tiny\(/\)} ([xshift=-2mm] product.north);
    \end{tikzpicture}
    \caption{Relations between various product topologies.}
    \label{fig:Comparisons}
\end{figure}

\subsection{The Vietoris Power on Subsets of Naturals}

We give in this section a thorough list of topological properties enjoyed by Vietoris powers of subsets of naturals,
and other related spaces.

\begin{lemma} \label{lem:BasicCompactness}
    For any \(s \in \omega^{<\omega}\), the set \([\![s]\!]\) as a subspace of \(\mathsf V(\omega^\omega)\)
    is homeomorphic to \(k^\omega\) with the Tychonoff product topology where \(k = \# \mathrm{img}(s)\),
    and hence is compact.
\end{lemma}
\begin{proof}
    First, note that \(k\) and \(\mathrm{img}(s)\), considered as discrete spaces, are homeomorphic since they
    are in bijective correspondence.
    Hence \(k^\omega\) and \(\mathrm{img}(s)^\omega\) are homeomorphic.
    Consider the mapping \(\phi : \mathrm{img}(s)^\omega \to \mathsf V(\omega^\omega)\) defined by \(\phi(f) = s^{\frown} f\).
    Note that \(\phi\) is a bijection \(\mathrm{img}(s)^\omega \to [\![s]\!]\).
    We finish by showing that \(\phi\) is a homeomorphism onto its range.
    Note that the set of
    \[[g]_n := \{ f \in \mathrm{img}(s)^\omega : f|_n = g|_n \},\]
    where \(n \in \omega\) and \(g \in \mathrm{img}(s)^\omega\), forms a basis
    for \(\mathrm{img}(s)^\omega\).
    Then note that \(\phi^{``}[g]_n = [\![s^{\frown} (g|_n)]\!]\) for any \(g \in \mathrm{img}(s)^\omega\).
    Since \(\phi\) is a bijection, this establishes that \(\phi\) is both an open and continuous map.
    This completes the proof.
\end{proof}

\begin{theorem} \label{thm:nOmegaProps}
    For any natural number \(n \geq 2\), \(\mathsf V(n^\omega)\) has the following properties:
    \begin{enumerate}[label=(\alph*)]
        \item \label{nOmegaSecondCount}
        \(\mathsf V(n^\omega)\) is second-countable.
        \item \label{nOmegaNotHomo}
        \(\mathsf V(n^\omega)\) is not homogeneous.
        \item \label{nOmegaZeroDim}
        \(\mathsf V(n^\omega)\) is zero-dimensional.
        \item \label{nOmegaSigmaComp}
        \(\mathsf V(n^\omega)\) is \(\sigma\)-compact.
        \item \label{nOmegaLocalCompact}
        \(\mathsf V(n^\omega)\) is locally compact.
    \end{enumerate}
    Consequently, \(\mathsf V(n^\omega)\) is separable, metrizable, Baire, and not a topological group.
\end{theorem}
\begin{proof}
    \ref{nOmegaSecondCount}: This follows from Remark \ref{rmk:DiscreteSpaceBasis}.

    \ref{nOmegaNotHomo}: Note that constant functions are isolated, and others are not by Proposition \ref{prop:NonDiscrete}.

    \ref{nOmegaZeroDim}: This follows immediately from Remark \ref{rmk:DiscreteSpaceBasis}.

    \ref{nOmegaSigmaComp}:
    By Lemma \ref{lem:BasicCompactness}, \([\![s]\!]\) is compact for any \(s \in n^{<\omega}\).
    Then, since \(\mathsf V(n^\omega) = \bigcup \{ [\![s]\!] : s \in n^{<\omega} \},\)
    we see that \(\mathsf V(n^\omega)\) is \(\sigma\)-compact.

    \ref{nOmegaLocalCompact}:
    For any \(f \in \mathsf V(n^\omega)\), let \(m \in \omega\) be large enough so that
    \(\mathrm{img}(f) = \mathrm{img}(f|_m)\).
    Then \([\![ f|_m ]\!]\) is a compact neighborhood of \(f\) by Lemma \ref{lem:BasicCompactness}.

    Note that \(\mathsf V(n^\omega)\) is separable as it is second-countable.
    Since \(\mathsf V(n^\omega)\) is a zero-dimensional Hausdorff space, \(\mathsf V(n^\omega)\) is completely regular.
    Moreover, since \(\mathsf V(n^\omega)\) is second-countable, \(\mathsf V(n^\omega)\) is metrizable by Urysohn's Metrization Theorem.
    \(\mathsf V(n^\omega)\) is Baire as it is a locally compact Hausdorff space.
    Lastly, we note that all topological groups are homogeneous, and, since \(\mathsf V(n^\omega)\)
    fails to be homogeneous, it cannot be a topological group.
\end{proof}

\begin{theorem} \label{thm:KOmegaOrd}
    The space \(\mathbb K(\omega, \mathrm{ord})\) has the following properties:
    \begin{enumerate}[label=(\alph*)]
        \item \label{KOmegaOrdSecondCountable}
        \(\mathbb K(\omega, \mathrm{ord})\) is second-countable.
        \item \label{KOmegaHomogeneous}
        \(\mathbb K(\omega, \mathrm{ord})\) is not homogeneous.
        \item \label{KOmegaOrdZeroDim}
        \(\mathbb K(\omega, \mathrm{ord})\) is zero-dimensional.
        \item \label{KOmegaOrdSigmaCompact}
        \(\mathbb K(\omega, \mathrm{ord})\) is \(\sigma\)-compact.
        \item \label{KOmegaOrdLocallyCompact}
        \(\mathbb K(\omega, \mathrm{ord})\) is locally compact.
        \item \label{KOmegaOrdComplete}
        \(\mathbb K(\omega, \mathrm{ord})\) is completely metrizable.
    \end{enumerate}
    Consequently, \(\mathbb K(\omega, \mathrm{ord})\) is separable, Baire, not a topological group, and embeds as a subspace of \(2^\omega\).
\end{theorem}
\begin{proof}
    Items \ref{KOmegaOrdSecondCountable}, \ref{KOmegaHomogeneous}, and \ref{KOmegaOrdZeroDim}
    follow in the same way as they did in Theorem \ref{thm:nOmegaProps}.

    \ref{KOmegaOrdSigmaCompact}:
    Note that \(\mathbb K(\omega, \mathrm{ord}) = \bigcup \{ [\![s]\!] : s \in \omega^{<\omega} \}\) and that each \([\![s]\!]\)
    is compact by Lemma \ref{lem:BasicCompactness}.

    \ref{KOmegaOrdLocallyCompact}:
    For any \(f \in \mathbb K(\omega, \mathrm{ord})\), let \(m \in \omega\) be large enough so that
    \(\mathrm{img}(f) = \mathrm{img}(f|_m)\).
    Then \([\![ f|_m ]\!]\) is a compact neighborhood of \(f\) by Lemma \ref{lem:BasicCompactness}.
    
    \ref{KOmegaOrdComplete}:
    It follows from \ref{KOmegaOrdSecondCountable}, \ref{KOmegaOrdZeroDim}, and Urysohn's Metrization Theorem
    that \(\mathbb K(\omega, \mathrm{ord})\) is metrizable, but we exhibit here a compatible metric on \(\mathbb K(\omega, \mathrm{ord})\) which is complete.
    For \(f,g \in \mathbb K(\omega, \mathrm{ord})\) with \(f \neq g\), let
    \[m(f,g) = \min\{ n \in \omega : f(n) \neq g(n) \}.\]
    Then define \(d : \mathbb K(\omega, \mathrm{ord})^2 \to [0,2]\) by
    \[d(f,g) = \begin{cases} 2, & \mathrm{img}(f) \neq \mathrm{img}(g),\\
        2^{-m(f,g)}, & \mathrm{img}(f) = \mathrm{img}(g) \wedge f \neq g,\\
        0, & f = g.
    \end{cases}\]
    Note that the equivalence of \(f = g\) and \(d(f,g) = 0\) follows immediately from the
    definition.
    It is also clear that \(d(f,g) = d(g,f)\) for all \(f, g \in \mathbb K(\omega, \mathrm{ord})\).
    We verify the triangle inequality by cases.
    So let \(f,g,h \in \mathbb K(\omega, \mathrm{ord})\).
    If \(\mathrm{img}(f) \neq \mathrm{img}(g)\) and \(\mathrm{img}(g) \neq \mathrm{img}(h)\),
    then \[d(f,h) = 2 \leq 4 = d(f,g) + d(g,h).\]
    If \(\mathrm{img}(f) = \mathrm{img}(h)\) and \(\mathrm{img}(f) \neq \mathrm{img}(g)\),
    then \[d(f,h) = 2^{-m(f,h)} \leq 1 \leq 2 = d(f,g) \leq d(f,g) + d(g,h).\]
    If \(\mathrm{img}(f) = \mathrm{img}(g)\) and \(\mathrm{img}(g) \neq \mathrm{img}(h)\),
    then \(d(f,h) = 2 = d(g,h) \leq d(f,g) + d(g,h)\).
    If \(\mathrm{img}(f) = \mathrm{img}(g) = \mathrm{img}(h)\), then the triangle inequality holds because
    the restriction of \(d\) in this context aligns with a standard compatible metric on \(\omega^\omega\).

    To see that this metric is compatible with \(\mathbb K(\omega, \mathrm{ord})\), first consider
    \([s,A]\) and \(f \in [s, A]\).
    Let \(n = \mathrm{len}(s)\).
    Note that \(B_d(f; 2^{-n-1}) \subseteq [s,A]\).
    On the other hand, consider \(B_d(f;\varepsilon)\) and \(g \in B_d(f;\varepsilon)\).
    If \(\varepsilon > 2\), \(g \in [\emptyset,\mathrm{img}(g)]\subseteq \mathbb K(\omega, \mathrm{ord}) = B_d(f;\varepsilon)\).
    Otherwise, \(\varepsilon \leq 2\).
    Then \(d(f,g) < \varepsilon \leq 2\) means that \(\mathrm{img}(f) = \mathrm{img}(g)\).
    Let \(n \in \omega\) be large enough such that \(\mathrm{img}(g|_n) = \mathrm{img}(g)\)
    and \(2^{-n} < \varepsilon - d(f,g)\).
    It follows that \(g \in [g|_n, \mathrm{img}(f)] \subseteq B_d(f; \varepsilon)\).
    Indeed, consider \(h \in [g|_n, \mathrm{img}(f)]\) and note that \(h|_n = g|_n\).
    Then \(m(g,h) \geq n\) and
    \[\mathrm{img}(g) = \mathrm{img}(g|_n) = \mathrm{img}(h|_n) \subseteq \mathrm{img}(h) \subseteq \mathrm{img}(g) = \mathrm{img}(f).\]
    It follows that
    \[d(f,h) = 2^{-m(f,h)} \leq 2^{-m(f,g)} + 2^{-m(g,h)} \leq d(f,g) + 2^{-n} < \varepsilon.\]
    
    The fact that \(d\) is complete follows from the fact that any \(d\)-Cauchy sequence must eventually
    consist of functions with equal finite image.
    Past that point, the functions will converge to the usual limit from \(\omega^\omega\).
\end{proof}

In the context of discrete spaces, Theorems \ref{thm:nOmegaProps} and \ref{thm:KOmegaOrd},
and Proposition \ref{prop:LargeDiscreteNotLindelof} fully characterize
when \(\mathbb K(X,\mathrm{ord})\) is Lindel\"{o}f, and even \(\sigma\)-compact.
We also show here that \(\mathbb K(X,\mathrm{ord})\) is not Lindel\"{o}f whenever \(X\) is an uncountable ordinal
with the usual order topology.
\begin{proposition}
    If \(\alpha\) is an uncountable ordinal, then \(\mathbb K(\alpha, \mathrm{ord})\) is not Lindel\"{o}f.
\end{proposition}
\begin{proof}
    We start by showing that \(\mathbb K(\omega_1, \mathrm{ord})\) is not Lindel\"{o}f.
    Consider \[\mathscr U := \{ [\beta; \emptyset, \emptyset] : \beta < \omega_1 \},\]
    an open cover of \(\mathbb K(\omega_1,\mathrm{ord})\).
    Note that \(\mathscr U\) has no countable subcover.

    Now suppose \(\alpha > \omega_1\) and note that
    \[\kappa := \sup\{ \# K : K \in K(\alpha) \} \geq \omega_1\]
    since \([0,\omega_1] \in K(\alpha)\).
    Then consider
    \[\mathscr U := \{ [\omega_1; \emptyset, \emptyset] \} \cup
    \left\{ \{ f \in \mathbb K(\alpha, \mathrm{ord}) : f(\beta) > \gamma \} : \langle \beta, \gamma \rangle \in \kappa \times \omega_1 \right\},\]
    an open cover of \(\mathbb K(\alpha, \mathrm{ord})\).
    Let \(\{ \langle \beta_n, \gamma_n \rangle : n \in \omega \} \subseteq \kappa \times \omega_1\) be arbitrary and let
    \[\mathscr V = \{ [\omega_1; \emptyset, \emptyset] \} \cup
    \left\{ \{ f \in \mathbb K(\alpha, \mathrm{ord}) : f(\beta_n) > \gamma_n \} : n \in \omega \right\}.\]
    Choose \(\beta_\ast \in \kappa \setminus \{ \beta_n : n \in \omega \}\) and define \(g : \kappa \to \alpha\) by
    \[g(\delta) = \begin{cases} 0, & \delta \neq \beta_\ast, \\ \omega_1, & \delta = \beta_\ast. \end{cases}\]
    Note that \(g \in \mathbb K(\alpha, \mathrm{ord}) \setminus \bigcup \mathscr V\).
    Consequently, \(\mathscr U\) has no countable subcover.
\end{proof}
\begin{question}
    Is \(\mathbb K(\alpha,\mathrm{ord})\) \(\sigma\)-compact or, less generally, Lindel\"{o}f for ordinals \(\alpha \in (\omega, \omega_1)\)?
    In particular, is \(\mathbb K(\omega + 1,\mathrm{ord})\) \(\sigma\)-compact?
\end{question}

Before we elaborate on properties of \(\mathsf V(\omega^\omega)\) we remind the reader of a few notions that will be relevant.

A family \(\mathcal A \subseteq [\omega]^\omega\) is said to be \emph{almost disjoint} if, for each pair of distinct
\(A, B \in \mathcal A\), \(A \cap B\) is finite.
It is well-known (see, for example, \cite[Lemma III.1.16]{Kunen2011}) that an almost disjoint family of cardinality
\(\mathfrak c\) exists.

For a linear order \(<\) on a set \(X\), we say that \(A \subseteq X\) is \emph{order-convex} if,
for \(a,b \in A\) and \(x \in X\), if \(a < x < b\), then \(x \in A\).
A \(T_1\) space \(X\) is a \emph{GO-space} (for \emph{generalized order space}) if there exists a linear order \(<\) on \(X\)
such that every point has a local basis consisting of order-convex sets.

Note that \(\mathbb K(\omega, \mathrm{ord})\) is a dense open subspace of \(\mathsf V(\omega^\omega)\) but,
despite the properties recorded in Theorem \ref{thm:KOmegaOrd}, \(\mathsf V(\omega^\omega)\) does not behave as nicely.
\begin{theorem} \label{thm:BairePropertyList}
    The space \(\mathsf V(\omega^\omega)\) has the following properties:
    \begin{enumerate}[label=(\alph*)]
        \item \label{VBaireNotHomo}
        \(\mathsf V(\omega^\omega)\) is not homogeneous.
        \item \label{VBaireFiniteRange}
        Every function with finite range has a local basis consisting of compact sets.
        \item \label{VBaireNotAllCompactNhoods}
        Functions with infinite range do not have compact neighborhoods.
        \item \label{VBaireNotLocallyCompact}
        \(\mathsf V(\omega^\omega)\) is not locally compact.
        \item \label{VBaireNotScattered}
        \(\mathsf V(\omega^\omega)\) is not scattered.
        \item \label{VBaireSep}
        \(\mathsf V(\omega^\omega)\) is separable.
        \item \label{VBaireFirstCount}
        \(\mathsf V(\omega^\omega)\) is first-countable.
        \item \label{VBaireZeroDim}
        \(\mathsf V(\omega^\omega)\) is zero-dimensional.
        \item \label{VBaireBaire}
        \(\mathsf V(\omega^\omega)\) is Baire.
        \item \label{VBaireWeight}
        \(w(\mathsf V(\omega^\omega)) = \mathfrak c\).
        \item \label{VBaireSpread}
        \(s(\mathsf V(\omega^\omega)) = \mathfrak c\).
        \item \label{VBaireExtent}
        \(e(\mathsf V(\omega^\omega)) = \mathfrak c\).
        \item \label{VBaireNotGO}
        \(\mathsf V(\omega^\omega)\) is not a GO-space.
        \item \label{VBaireNotMenger}
        \(\mathsf V(\omega^\omega)\) is not Menger.
        \item \label{VBaireSquare}
        There is a continuous bijection \(\mathsf V(\omega^\omega) \to \mathsf V(\omega^\omega)^2\).
    \end{enumerate}
    Consequently, \(\mathsf V(\omega^\omega)\) is completely regular, not \(\sigma\)-compact, not metrizable, not hereditarily separable,
    and not hereditarily Lindel\"{o}f.
\end{theorem}
\begin{proof}
    \ref{VBaireNotHomo}:
    As in Theorem \ref{thm:nOmegaProps}, constant functions in \(\mathsf V(\omega^\omega)\) are isolated, and others are not by Proposition \ref{prop:NonDiscrete}.

    \ref{VBaireFiniteRange}:
    Let \(f \in \omega^\omega\) have finite range and \(n \in \omega\) be large enough so that \(\mathrm{img}(f) = \mathrm{img}(f|_n)\).
    Note then that \(\{ [\![ f|_{n+k} ]\!] : k \in \omega \}\) is a local basis at \(f\) consisting of compact sets by
    Lemma \ref{lem:BasicCompactness}.

    \ref{VBaireNotAllCompactNhoods}:
    Let \(f \in \omega^\omega\) have infinite range and consider, for \(n \in \omega\),
    \(U_n := [f|_{n+1}, \mathrm{img}(f)]\), a closed neighborhood of \(f\).
    We show that \(U_n\) is not compact.
    Consider \[\mathscr U := \{ [(f|_{n+1})^{\frown} s, \mathrm{img}(f)] : s \in \mathrm{img}(f)^{<\omega} \},\]
    an open cover of \(U_n\).
    Then let \(s_1,\ldots, s_m \in \mathrm{img}(f)^{<\omega}\) be arbitrary.
    It is evident that we can find a function \(g \in \omega^\omega\) such that
    \[g \in U_n \setminus \bigcup_{j=1}^m [(f|_{n+1})^{\frown} s_j, \mathrm{img}(f)].\]
    Thus, \(U_n\) is not compact.
    Note, consequently, that any neighborhood of \(f\) must contain some \(U_n\), so no neighborhood of \(f\) is compact.

    \ref{VBaireNotLocallyCompact}:
    This follows immediately from \ref{VBaireNotAllCompactNhoods}.

    \ref{VBaireNotScattered}:
    This follows immediately from Proposition \ref{prop:NonDiscrete}.
    
    \ref{VBaireSep}:
    This follows immediately from Theorem \ref{thm:Density}, but also very directly by noting that
    the set of eventually constant sequences form a countable dense subset of \(\mathsf V(\omega^\omega)\).

    \ref{VBaireFirstCount}:
    Given any \(f \in \omega^\omega\), note that \(\{ [ f|_{n+1}, \mathrm{img}(f) ] : n \in \omega \}\) is a local basis at \(f\)
    in \(\mathsf V(\omega^\omega)\).

    \ref{VBaireZeroDim}:
    This follows immediately by Remark \ref{rmk:DiscreteSpaceBasis}.

    \ref{VBaireBaire}:
    By Theorem \ref{thm:KOmegaOrd}, we know that \(\mathbb K(\omega, \mathrm{ord})\) is Baire.
    It is also evident that \(\mathbb K(\omega, \mathrm{ord})\) is a dense subspace of \(\mathsf V(\omega^\omega)\).
    It follows that \(\mathsf V(\omega^\omega)\) is Baire.
    
    \ref{VBaireWeight}:
    It is immediate that \(w(\mathsf V(\omega^\omega)) \leq \mathfrak c\) since the basis defined
    in Remark \ref{rmk:DiscreteSpaceBasis} has a cardinality that is bounded by \(\mathfrak c\) in this case.
    So consider \(\kappa < \mathfrak c\) and a \(\kappa\)-sized collection
    of basic open subsets of \(\mathsf V(\omega^\omega)\),
    \[
        \{ [s_\alpha, A_\alpha] : s_\alpha \in \omega^{<\omega}, A_\alpha \subseteq \omega, \alpha < \kappa \}.
    \]
    Let \(B \subseteq \omega\) be such that \(B \neq A_\alpha\) for any \(\alpha < \kappa\). Note that \(B^\omega\) is an open set in \(\mathsf V(\omega^\omega)\).
    We will show that
    \[B^\omega \neq \bigcup \{ [s_\alpha, A_\alpha] : \alpha < \kappa \wedge [s_\alpha, A_\alpha] \subseteq B^\omega \}.\]

    For any \(\alpha < \kappa\), if \([s_\alpha, A_\alpha] \subseteq B^\omega\), then it must be the case that
    \(A_\alpha \subseteq B\).
    Indeed, let \(f \in \omega^\omega\) be any function which surjects onto \(A_\alpha\) that agrees with
    \(s_\alpha\).
    Then, \(A_\alpha = \mathrm{img}(f) \subseteq B\).
    It is also immediate that, if \(A_\alpha \subseteq B\), then \([s_\alpha, A_\alpha] \subseteq B^\omega\).
    Hence,
    \[\bigcup \{ [s_\alpha, A_\alpha] : \alpha < \kappa \wedge [s_\alpha, A_\alpha] \subseteq B^\omega \}
    = \bigcup \{ [s_\alpha, A_\alpha] : \alpha < \kappa \wedge A_\alpha \subseteq B \}.\]
    
    For each \(\alpha < \kappa\) for which \(A_\alpha \subseteq B\), \(A_\alpha\subsetneq B\) since \(B \neq A_\alpha\).
    So let \(x_\alpha \in B \setminus A_\alpha\).
    Note that \[E := \{ x_\alpha : \alpha < \kappa \wedge A_\alpha \subseteq B \} \subseteq \omega,\]
    so there is a surjection \(f : \omega \to E\).
    Note that \(f \in B^\omega\) but \(f \not\in [s_\alpha, A_\alpha]\) for any of the \(\alpha < \kappa\) for which
    \(A_\alpha \subseteq B\).

    \ref{VBaireSpread}:
    Let \(\mathcal A \subseteq [\omega]^\omega\) be such that, for every \(\{ A, B \} \in [\mathcal A]^2\),
    \(A \setminus B \neq \emptyset\) and \(B \setminus A \neq \emptyset\).
    Note that any almost disjoint family satisfies the required condition.
    For each \(A \in \mathcal A\), let \(f_A \in \omega^\omega\) be such that \(\mathrm{img}(f_A) = A\).
    Note that \(f_A \in [\emptyset, A]\).
    Also note that \(D := \{ f_A : A \in \mathcal A \}\) is relatively discrete since
    \[[\emptyset, A] \cap D = \{f_A\}.\]
    Since it is known that an almost disjoint family of size \(\mathfrak c\) exists, we can thus
    conclude that \[s(\mathsf V(\omega^\omega)) = \mathfrak c.\]

    \ref{VBaireExtent}:
    A closed and discrete set of cardinality \(\mathfrak c\) is produced in
    \cite[Example 2.4]{PinchedCube}.

    \ref{VBaireNotGO}:
    By \cite[Prop. 2.10(a)]{LutzerOrder}, every separable GO-space is hereditarily separable.
    Since \(\mathsf V(\omega^\omega)\) has uncountable spread by \ref{VBaireSpread}, it is not hereditarily separable.
    Hence, since \(\mathsf V(\omega^\omega)\) is separable, it follows that \(\mathsf V(\omega^\omega)\) is not a GO-space.

    \ref{VBaireNotMenger}:
    Consider, for each \(n \in \omega\), the cover \[\mathscr U_n = \left\{ [ s , \omega ] : s \in \omega^{n+1} \right\}.\]
    Then consider, for every \(n \in \omega\), a collection
    \[\mathcal F_n := \{ [s_{n,1} , \omega] , \ldots , [s_{n,m_n}, \omega] \} \in \left[ \mathscr U_n \right]^{<\aleph_0}.\]
    For \(k \in \omega\), suppose we've defined
    \(\langle x_j : j < k \rangle\).
    Since \[A_k := \bigcup_{j \leq k} \bigcup_{\ell=1}^{m_j} \mathrm{img}(s_{j,\ell})\]
    is finite, we can let \(x_k \in \omega \setminus A_k\).
    This defines \(\langle x_j : j \leq k \rangle\).
    
    Note that \[\langle x_k : k \in \omega \rangle \not\in \bigcup_{n\in\omega} \mathcal F_n.\]
    Conclusively, \(\mathsf V(\omega^\omega)\) is not Menger.

    \ref{VBaireSquare}:
    Let \(\beta : \omega \to \omega^2\) be a bijection and let \(\cdot_j : \omega^2 \to \omega\),
    \(j = 1,2\), be the standard coordinate projection map;
    that is, \(\beta(n) = \langle \beta(n)_1, \beta(n)_2 \rangle\).
    Then, for \(j=1,2\), let \(\phi_j : \omega^\omega \to \omega^\omega\)
    be defined by \(\phi_j(f)(n) = \beta(f(n))_j\).
    Then we define \(\Phi : \mathsf V(\omega^\omega) \to \mathsf V(\omega^\omega)^2\) by the rule
    \(\Phi(f) = \langle \phi_1(f) , \phi_2(f) \rangle\).
    A routine argument shows that \(\Phi\) is a bijection.

    We know establish that \(\Phi\) is continuous.
    So suppose \(f \in \Phi^{-1}[[s_1, A_1] \times [s_2,A_2]]\) where \(s_1, s_2 \in \omega^{<\omega}\) and \(A_1,A_2 \subseteq \omega\).
    Let \(m = \max \{ \mathrm{len}(s_1), \mathrm{len}(s_2) \}\).
    Then, it can be shown that \[f \in [f|_m , \beta^{-1}[A_1 \times A_2]] \subseteq \Phi^{-1}[[s_1, A_1] \times [s_2,A_2]].\]
    This establishes that \(\Phi\) is continuous.

    We see that \(\mathsf V(\omega^\omega)\) is completely regular as a zero-dimensional Hausdorff space and that \(\mathsf V(\omega^\omega)\)
    is not \(\sigma\)-compact as it is not Menger.
    Finally, since the weight and density for metrizable spaces agree (see \cite[Theorem 8.1]{Hodel}), and we have that
    \(d(\mathsf V(\omega^\omega)) = \omega \neq \mathfrak c = w(\mathsf V(\omega^\omega))\),
    \(\mathsf V(\omega^\omega)\) is not metrizable.
\end{proof}

Given the properties outlined above, there are three topological spaces in the \(\pi\)-base \cite{PiBase} that naturally compare to
\(\mathsf V(\omega^\omega)\) or \(\mathsf V(\omega^\omega)^\prime\), where \(\mathsf V(\omega^\omega)^\prime\) denotes \(\mathsf V(\omega^\omega)\)
with its isolated points removed:
the Baire space (\cite[\href{https://topology.pi-base.org/spaces/S000028}{S28}]{PiBase}, \cite[Space 31]{Counterexamples}),
the Sorgenfrey line (\cite[\href{https://topology.pi-base.org/spaces/S000043}{S43}]{PiBase}, \cite[Space 51]{Counterexamples}),
and the rational sequence topology (\cite[\href{https://topology.pi-base.org/spaces/S000057}{S57}]{PiBase}, \cite[Space 65]{Counterexamples}).
In the following, we demonstrate that
neither \(\mathsf V(\omega^\omega)\) nor \(\mathsf V(\omega^\omega)^\prime\) are homeomorphic to any of these spaces.

It is evident that every constant function in \(\mathsf V(\omega^\omega)\) is isolated.
Moreover, by Proposition \ref{prop:NonDiscrete}, the only isolated points of \(\mathsf V(\omega^\omega)\) are the constant functions.
So \(\mathsf V(\omega^\omega)^\prime\) is the subspace of \(\mathsf V(\omega^\omega)\)
consisting of non-constant functions.
Since \(\mathsf V(\omega^\omega)^\prime\) is a clopen subspace of \(\mathsf V(\omega^\omega)\) with a countable complement, we see that
\(w(\mathsf V(\omega^\omega)^\prime) = w(\mathsf V(\omega^\omega)) = \mathfrak c\),
and so \(\mathsf V(\omega^\omega)^\prime\) is not homeormorphic to the Baire space \(\omega^\omega\).

Recall that the Sorgenfrey line is separable, first-countable, zero-dimensional, Baire, of weight \(\mathfrak c\), and not Menger
(see \cite[Space 51]{Counterexamples} and \cite[Lemma 17]{CombinatoricsOfOpenCoversIX}).
Again, since \(\mathsf V(\omega^\omega)\) has isolated points and the Sorgenfrey line doesn't, these two spaces clearly cannot be homeomorphic.
Even more, \(\mathsf V(\omega^\omega)^\prime\) is not homeomorphic to the Sorgenfrey line.
One of the reasons for this, which we record here, relates to compact subsets.
Indeed, it is known that the Sorgenfrey line contains no uncountable
compact sets (see \cite[Space 51]{Counterexamples}).
Since \([\![ \langle 0, 1 \rangle ]\!] \subseteq \mathsf V(\omega^\omega)^\prime\) is an uncountable compact subset by
Lemma \ref{lem:BasicCompactness},
we see that \(\mathsf V(\omega^\omega)^\prime\) and the Sorgenfrey line cannot be homeomorphic.

A rational sequence topology is defined as follows. For each irrational \(x \in \mathbb R\), fix a sequence \(\langle x_i : i \in \omega \rangle\)
of rational numbers \(x_i \to x\).
The corresponding rational sequence topology on \(\mathbb R\) is defined by declaring each rational point open and letting
\(U_n(x) = \{x_i : i > n\} \cup \{x\}\) be a local basis at each irrational \(x\).
The rational sequence topology is scattered (see \cite[Space 65]{Counterexamples}), but neither \(\mathsf V(\omega^\omega)\) nor
\(\mathsf V(\omega^\omega)^\prime\) are scattered by Proposition \ref{prop:NonDiscrete}.
Therefore, any rational sequence topology on \(\mathbb R\) is not homeomorphic to either \(\mathsf V(\omega^\omega)\)
or \(\mathsf V(\omega^\omega)^\prime\).

In Table \ref{table:PropertyTables}, we summarize some of the properties discussed herein
where \(\mathsf b(\omega^\omega)\) denotes \(\omega^\omega\) endowed with the box topology, \(\mathbb R_\ell\) denotes the Sorgenfrey line,
and \(\mathbb R_{\mathrm{rs}}\) denotes a rational sequence topology.
Note that \(\mathsf b(\omega^\omega)\) is homeomorphically equivalent to \(D(\mathfrak c)\).
Of course, \ding{51} indicates that the indicated space satisfies the indicated property and \ding{55} indicates
that the indicated space doesn't satisfy the indicated property.
Note that the asserted properties for \(\mathsf V(\omega^\omega)^\prime\) follow from Theorem \ref{thm:BairePropertyList},
along with the fact that \(\mathsf V(\omega^\omega)^\prime\) is a clopen subspace of \(\mathsf V(\omega^\omega)\) with
a countable relatively discrete complement.
For the assertion or refutation of particular properties for particular spaces not addressed explicitly herein,
we refer the reader to \cite{PiBase}, along with the relevant references therein, and \cite{Counterexamples}.
\begin{table}[h]
    \begin{tabular}{l|>{\centering\arraybackslash}p{3em}|>{\centering\arraybackslash}p{3em}|>{\centering\arraybackslash}p{3em}|>{\centering\arraybackslash}p{3em}|>{\centering\arraybackslash}p{3em}|>{\centering\arraybackslash}p{3em}}
        & \(\mathsf b(\omega^\omega)\) & \(\omega^\omega\) &
        \(\mathsf V(\omega^\omega)\) & \(\mathsf V(\omega^\omega)^\prime\)
        & \(\mathbb R_\ell\) & \(\mathbb R_{\mathrm{rs}}\)\\
        \hline
        \hline
        first-countable     & \ding{51} & \ding{51} & \ding{51} & \ding{51} & \ding{51} & \ding{51}\\
        \hline
        second-countable    & \ding{55} & \ding{51} & \ding{55} & \ding{55} & \ding{55} & \ding{55}\\
        \hline
        separable           & \ding{55} & \ding{51} & \ding{51} & \ding{51} & \ding{51} & \ding{51}\\
        \hline
        countable spread    & \ding{55} & \ding{51} & \ding{55} & \ding{55} & \ding{51} & \ding{55}\\
        \hline
        zero-dimensional    & \ding{51} & \ding{51} & \ding{51} & \ding{51} & \ding{51} & \ding{51}\\
        \hline
        Baire               & \ding{51} & \ding{51} & \ding{51} & \ding{51} & \ding{51} & \ding{51}\\
        \hline
        crowded             & \ding{55} & \ding{51} & \ding{55} & \ding{51} & \ding{51} & \ding{55}\\
        \hline
        scattered           & \ding{51} & \ding{55} & \ding{55} & \ding{55} & \ding{55} & \ding{51}\\
        \hline
        GO-space            & \ding{51} & \ding{51} & \ding{55} & \ding{55} & \ding{51} & \ding{55}\\
        \hline
        homogeneous          & \ding{51} & \ding{51} & \ding{55} & \ding{55} & \ding{51} & \ding{55}\\
        \hline
        locally compact     & \ding{51} & \ding{55} & \ding{55} & \ding{55} & \ding{55} & \ding{51}\\
        \hline
        Menger              & \ding{55} & \ding{55} & \ding{55} & \ding{55} & \ding{55} & \ding{55}\\
        \hline
        Lindel\"{o}f        & \ding{55} & \ding{51} & \ding{55} & \ding{55} & \ding{51} & \ding{55}\\
        \hline
        normal              & \ding{51} & \ding{51} & \ding{55} & \ding{55} & \ding{51} & \ding{55}\\
        \hline
        countable extent    & \ding{55} & \ding{51} & \ding{55} & \ding{55} & \ding{51} & \ding{55}\\
        \hline
    \end{tabular}
    \caption{Comparison of particular topologies on the continuum.}
    \label{table:PropertyTables}
\end{table}

\subsection{Commentary on Covering Games}

We present a particular version of Corollary 2.17 from \cite{CHVietoris} here for simplicity and immediate relevance.
In particular, we remove the generality where the translation functions \(\overleftarrow{\mathrm{T}}_{\mathrm{I}}\) and
\(\overrightarrow{\mathrm{T}}_{\mathrm{II}}\) can vary depending on the inning \(n \in \omega\).
\begin{theorem}[{\cite[Cor. 2.17]{CHVietoris}}] \label{TranslationTheorem}
    Let \(\mathcal A\), \(\mathcal B\), \(\mathcal C\), and \(\mathcal D\) be collections.
    Suppose there are functions
    \begin{itemize}
        \item \(\overleftarrow{\mathrm{T}}_{\mathrm{I}} :\mathcal B \to \mathcal A\) and
        \item \(\overrightarrow{\mathrm{T}}_{\mathrm{II}} : \left(\bigcup \mathcal A \right) \times \mathcal B \to \bigcup \mathcal B\)
    \end{itemize}
    such that the following two properties hold:
    \begin{enumerate}[label=(T\arabic*)]
        \item \label{translationPropI}
        If \(x \in \overleftarrow{\mathrm{T}}_{\mathrm{I}}(B)\), then
        \(\overrightarrow{\mathrm{T}}_{\mathrm{II}}(x,B) \in B\).
        \item \label{translationPropII}
        If \(\mathcal F_n \in \left[\overleftarrow{\mathrm{T}}_{\mathrm{I}}(B_n)\right]^{<\aleph_0}\) for every \(n \in \omega\) and
        \(\bigcup_{n \in \omega} \mathcal F_n \in \mathcal C\), then
        \[\bigcup_{n \in \omega} \left\{ \overrightarrow{\mathrm{T}}_{\mathrm{II}}(x,B_n) : x \in \mathcal F_n \right\} \in \mathcal D.\]
    \end{enumerate}
    Then, for \(\ast \in \{1,\mathrm{fin}\}\), \(\mathsf G_\ast(\mathcal A,\mathcal C) \leq^+_{\mathrm{II}} \mathsf G_\ast(\mathcal B, \mathcal D)\).
\end{theorem}
Note that, in view of properties \ref{translationPropI} and \ref{translationPropII}, \(\overrightarrow{\mathrm{T}}_{\mathrm{II}}\) need not
have full domain in the first coordinate.
That is, for \(B \in \mathcal B\), we will only concern ourselves with defining
\(\overrightarrow{\mathrm{T}}_{\mathrm{II}}(x,B)\) for each \(x \in \overleftarrow{\mathrm{T}}_{\mathrm{I}}(B)\).

\begin{proposition} \label{prop:ContinuousImage}
    If \(Y\) is a continuous image of \(X\), then,
    for \(\ast \in \{1, \mathrm{fin} \}\),
    \[\mathsf G_\ast(\mathcal O_X,\mathcal O_X) \leq_{\mathrm{II}} \mathsf G_\ast(\mathcal O_Y, \mathcal O_Y)\]
    and
    \[\mathsf G_\ast(\Omega_X,\Omega_X) \leq_{\mathrm{II}} \mathsf G_\ast(\Omega_Y, \Omega_Y).\]
\end{proposition}
\begin{proof}
    Suppose \(f:X \to Y\) is a continuous surjection.
    We start with the assertion for the topological operator \(\mathcal O\).
    For an open cover \(\mathscr U\) of \(Y\), let
    \[\overleftarrow{\mathrm{T}}_{\mathrm{I}}(\mathscr U) = \{ f^{-1}(U) : U \in \mathscr U\}.\]
    Note that \(\overleftarrow{\mathrm{T}}_{\mathrm{I}}(\mathscr U)\) is an open cover of \(X\).
    For \(\mathscr U \in \mathcal O_Y\) and each \(V \in \overleftarrow{\mathrm{T}}_{\mathrm{I}}(\mathscr U)\),
    choose \(\overrightarrow{\mathrm{T}}_{\mathrm{II}}(V, \mathscr U) \in \mathscr U\) to be such that
    \(V = f^{-1}\left(\overrightarrow{\mathrm{T}}_{\mathrm{II}}(V, \mathscr U)\right)\).

    Now suppose we have a sequence \(\langle \mathscr U_n : n \in \omega \rangle\) of open covers of \(Y\)
    and a sequence
    \(\langle \mathscr G_n : n \in \omega \rangle\) with the properties that
    \[\mathscr G_n \in \left[ \overleftarrow{\mathrm{T}}_{\mathrm{I}}(\mathscr U_n)\right]^{<\aleph_0}\]
    for each \(n \in \omega\) and \(\bigcup_{n\in\omega} \mathscr G_n\) is a cover of \(X\).
    We show that \[\bigcup_{n\in\omega}\left\{ \overrightarrow{\mathrm{T}}_{\mathrm{II}}(V, \mathscr U_n) : V \in \mathscr G_n \right\} \in \mathcal O_Y.\]
    For a given \(y \in Y\), let \(x \in f^{-1}(y)\) and choose \(n \in \omega\) and
    \(V\in\mathscr G_n\) such that \(x \in V\).
    It follows that \(y \in \overrightarrow{\mathrm{T}}_{\mathrm{II}}(V, \mathscr U_n)\).

    Conclusively, Theorem \ref{TranslationTheorem} applies to obtain that
    \[\mathsf G_\ast(\mathcal O_X,\mathcal O_X) \leq_{\mathrm{II}} \mathsf G_\ast(\mathcal O_Y, \mathcal O_Y).\]

    Now we address the assertion for the topological operator \(\Omega\).
    For \(\mathscr U \in \Omega_Y\), let, as before,
    \[\overleftarrow{\mathrm{T}}_{\mathrm{I}}(\mathscr U) = \{ f^{-1}(U) : U \in \mathscr U\}.\]
    To see that \(\overleftarrow{\mathrm{T}}_{\mathrm{I}}(\mathscr U) \in \Omega_X\), let
    \(F \in [X]^{<\aleph_0}\) and note that \(f[F] \in [Y]^{<\aleph_0}\).
    Then there is \(U \in \mathscr U\) such that \(f[F] \subseteq U\).
    It follows that \(F \subseteq f^{-1}(U)\).
    Hence, \(\overleftarrow{\mathrm{T}}_{\mathrm{I}}(\mathscr U) \in \Omega_X\).

    For \(\mathscr U \in \Omega_Y\) and each \(V \in \overleftarrow{\mathrm{T}}_{\mathrm{I}}(\mathscr U)\),
    choose \(\overrightarrow{\mathrm{T}}_{\mathrm{II}}(V, \mathscr U) \in \mathscr U\) to be such that
    \(V = f^{-1}\left(\overrightarrow{\mathrm{T}}_{\mathrm{II}}(V, \mathscr U)\right)\).

    Now suppose we have a sequence \(\langle \mathscr U_n : n \in \omega \rangle\) of \(\omega\)-covers of \(Y\)
    and a sequence
    \(\langle \mathscr G_n : n \in \omega \rangle\) with the properties that
    \[\mathscr G_n \in \left[ \overleftarrow{\mathrm{T}}_{\mathrm{I}}(\mathscr U_n)\right]^{<\aleph_0}\]
    for each \(n \in \omega\) and \(\bigcup_{n\in\omega} \mathscr G_n \in \Omega_X\).
    We show that \[\bigcup_{n\in\omega}\left\{ \overrightarrow{\mathrm{T}}_{\mathrm{II}}(V, \mathscr U_n) : V \in \mathscr G_n \right\} \in \Omega_Y.\]
    For a given \(F \in [Y]^{<\aleph_0}\), choose \(x_y \in f^{-1}(y)\) for each \(y \in F\).
    Since \(\{ x_y : y \in F \} \in [X]^{<\aleph_0}\), there is some \(n \in \omega\) and
    \(V \in \mathscr G_n\) with \(\{ x_y : y \in F \} \subseteq V\).
    It follows that \(F \subseteq \overrightarrow{\mathrm{T}}_{\mathrm{II}}(V, \mathscr U_n)\).

    Conclusively, Theorem \ref{TranslationTheorem} applies to obtain that
    \[\mathsf G_\ast(\Omega_X,\Omega_X) \leq_{\mathrm{II}} \mathsf G_\ast(\Omega_Y, \Omega_Y).\]
    This finishes the proof.
\end{proof}

\begin{definition}
    Let \(X\) and \(Y\) be topological spaces. Then a continuous \(f:X \to Y\) is
    \begin{itemize}
        \item \emph{proper} if whenever \(L \subseteq Y\) is compact, \(f^{-1}[L]\) is compact, and
        \item a \emph{compact covering map} if whenever \(L \subseteq Y\) is compact, there is a compact \(K \subseteq X\) so that \(f[K] = L\).
    \end{itemize}
    If there is a compact covering map from \(X\) to \(Y\), we say that \emph{\(X\) is a compact covering of \(Y\)}.
\end{definition}
\begin{remark}
    Note that every proper map is a compact covering map.
\end{remark}

\begin{proposition} \label{prop:ProperImage}
    If \(X\) is a compact covering of \(Y\), then
    \[\mathsf G_\ast(\mathcal K_X,\mathcal K_X) \leq_{\mathrm{II}} \mathsf G_\ast(\mathcal K_Y, \mathcal K_Y),\]
    for \(\ast \in \{ 1, \mathrm{fin} \}\).
\end{proposition}
\begin{proof}
    Suppose \(f:X \to Y\) is a compact covering map.
    For a \(k\)-cover \(\mathscr U\) of \(Y\), let
    \(\overleftarrow{\mathrm{T}}_{\mathrm{I}}(\mathscr U) = \{ f^{-1}(U) : U \in \mathscr U\}\).
    We claim that \(\overleftarrow{\mathrm{T}}_{\mathrm{I}}(\mathscr U)\) is a \(k\)-cover of \(X\).
    Suppose \(K \subseteq X\) is compact.
    Then \(f[K]\) is compact, so there is some \(U \in \mathscr U\) such that \(f[K] \subseteq U\).
    It follows that \(K \subseteq f^{-1}(U)\).
    
    Now, for \(\mathscr U \in \mathcal K_Y\) and each \(V \in \overleftarrow{\mathrm{T}}_{\mathrm{I}}(\mathscr U)\),
    choose \(\overrightarrow{\mathrm{T}}_{\mathrm{II}}(V, \mathscr U) \in \mathscr U\) to be such that
    \(V = f^{-1}\left(\overrightarrow{\mathrm{T}}_{\mathrm{II}}(V, \mathscr U)\right)\).

    Suppose we have a sequence \(\langle \mathscr U_n : n \in \omega \rangle\) of \(k\)-covers of \(Y\)
    and a sequence
    \(\langle \mathscr G_n : n \in \omega \rangle\) with the properties that
    \[\mathscr G_n \in \left[ \overleftarrow{\mathrm{T}}_{\mathrm{I}}(\mathscr U_n)\right]^{<\aleph_0}\]
    for each \(n \in \omega\) and \(\bigcup_{n\in\omega} \mathscr G_n\) is a \(k\)-cover of \(X\).
    We show that \[\bigcup_{n\in\omega}\left\{ \overrightarrow{\mathrm{T}}_{\mathrm{II}}(V, \mathscr U_n) : V \in \mathscr G_n \right\} \in \mathcal K_Y.\]
    So let \(L \subseteq Y\) be compact.
    Since \(f\) is a compact covering map, there is a compact \(K \subseteq X\) such that \(f[K] = L\).
    Let \(n \in \omega\) and \(V \in \mathscr G_n\) be such that \(K \subseteq V\).
    It follows that \(L \subseteq \overrightarrow{\mathrm{T}}_{\mathrm{II}}(V, \mathscr U_n)\).

    Conclusively, Theorem \ref{TranslationTheorem} applies.
\end{proof}
In the context of Proposition \ref{prop:ProperImage}, some hypotheses are necessary to guarantee the conclusion.
\begin{example}
    Note that \(\omega\) is a \(k\)-Rothberger space but the rationals \(\mathbb Q\) are not \(k\)-Menger
    (see \cite[Example 5.4 (3)]{TraditionalMenger}).
    Since any enumeration \(\omega \to \mathbb Q\) is a continuous surjection, we see that
    the conclusion of Proposition \ref{prop:ProperImage} does not hold for arbitrary continuous surjections.
\end{example}
We provide here a sufficient condition for a map to be a compact covering map.
\begin{proposition} \label{prop:OpenSurjection}
    Suppose \(f : X \to Y\) has the property that there exists a continuous \(g : Y \to X\) such that
    \(g(y) \in f^{-1}(y)\) for each \(y \in Y\).
    Then, for any compact \(L \subseteq Y\), there exists a compact \(K \subseteq X\) such that \(f[K] = L\).
    Note that any open surjection has the desired property.
    Hence, if \(f\) is a continuous open surjection, then \(f\) is a compact covering map.
\end{proposition}
\begin{proof}
    Suppose \(g  : Y \to X\) is a continuous map such that \(g(y) \in f^{-1}(y)\) for each \(y \in Y\).
    Thus, for any compact \(L \subseteq Y\), \(g[L]\) is a compact subset of \(X\) and
    \(f \circ g[L] = L\).
    In the case that \(f : X \to Y\) is an open surjection, any choice function \(g : Y \to X\)
    with \(g(y) \in f^{-1}(y)\) for each \(y\in Y\) is continuous.
\end{proof}

\subsection{Comparing the Topologies of Ordered and Unordered Compact Sets}

\begin{proposition} \label{prop:ImageMappingQuotient}
    The natural mapping \(f \mapsto \mathrm{img}(f)\), \(\mathbb K(X,\mathrm{ord},\kappa) \to \mathbb K(X)\), is continuous.
    When \(\kappa\) is an infinite cardinal, then it is also an open mapping onto its range,
    and thus a quotient mapping onto its range.
\end{proposition}
\begin{proof}
    Let \(R = \{ \mathrm{img}(f) : f \in \mathbb K(X,\mathrm{ord},\kappa) \}\).
    We start by showing \(f \mapsto \mathrm{img}(f)\), \(\mathbb K(X, \mathrm{ord}, \kappa) \to R\),
    is continuous.
    Consider \([U_1, \ldots, U_n]\) such that \([U_1 , \ldots, U_n] \cap R \neq \emptyset\).
    Then consider \(f \in \mathbb K(X,\mathrm{ord},\kappa)\) with \(\mathrm{img}(f) \in [U_1 , \ldots, U_n]\).
    For each \(j \in \{ 1,\ldots, n \}\), let \(\alpha_j \in \kappa\) be such that \(f(\alpha_j) \in U_j\).
    Then let \(U = \bigcup_{j=1}^n U_j\), \(\Lambda = \{ \alpha_1, \ldots, \alpha_n \}\), and define
    \(V : \Lambda \to \mathscr T_X\) by the rule
    \(V_{\alpha_j} = U_j\).
    Note that \(f \in [U; \Lambda, V]\) and, if \(g \in [U; \Lambda, V] \cap \mathbb K(X,\mathrm{ord}, \kappa)\), then \(\mathrm{img}(g) \in [U_1, \ldots, U_n]\).
    This establishes continuity.
    
    Now suppose \(\kappa\) is infinite.
    To see that \(f \mapsto \mathrm{img}(f)\), \(\mathbb K(X, \mathrm{ord}, \kappa) \to R\), is an open map,
    we show that \[\mathrm{img}^{``}([U; \Lambda, V] \cap \mathbb K(X, \mathrm{ord}, \kappa)) = [W : W \in \mathscr F] \cap R,\]
    where \(\mathscr F = \{ U \} \cup \{ U\cap V_\alpha : \alpha \in \Lambda \}\).
    Evidently, if \(f \in [U; \Lambda, V] \cap \mathbb K(X, \mathrm{ord}, \kappa)\), then \(\mathrm{img}(f) \in [W : W \in \mathscr F] \cap R\).
    Hence, \[\mathrm{img}^{``}([U; \Lambda, V] \cap \mathbb K(X, \mathrm{ord}, \kappa)) \subseteq [W : W \in \mathscr F] \cap R.\]
    So let \(K \in [W : W \in \mathscr F] \cap R\).
    Since \(K \in R\), we know that \(\# K \leq \kappa\).
    For \(\alpha \in \Lambda\), let \(x_\alpha \in K \cap U \cap V_\alpha\).
    If \(K \setminus \{ x_\alpha : \alpha \in \Lambda \} = \emptyset\),
    then we can choose \(y \in K\) and define \(f : \kappa \to X\)
    by the rule
    \[f(\alpha) = \begin{cases} x_\alpha, & \alpha \in \Lambda,\\ y, & \alpha \in \kappa \setminus \Lambda. \end{cases}\]
    Otherwise, let \(f_0 : \kappa \setminus \Lambda \to K \setminus \{ x_\alpha : \alpha \in \Lambda \}\) be a surjection
    and define \(f: \kappa \to X\) by the rule
    \[f(\alpha) = \begin{cases} x_\alpha, & \alpha \in \Lambda,\\ f_0(\alpha), & \alpha \in \kappa \setminus \Lambda. \end{cases}\]
    Note that \(\mathrm{img}(f) = K\) and that \(f \in [U; \Lambda, V]\).
    Hence, \(K \in \mathrm{img}^{``}[U;\Lambda, V]\).
    Thus, \[[W : W \in \mathscr F] \cap R \subseteq \mathrm{img}^{``}[U;\Lambda, V],\]
    finishing the proof.
\end{proof}
The requirement in Proposition \ref{prop:ImageMappingQuotient} that \(\kappa\) be infinite for the referenced mapping
to be open onto its range is, in general, necessary.
\begin{example}
    If \(\kappa\) is finite, then the mapping \(f \mapsto \mathrm{img}(f)\), \(\mathbb K(X, \mathrm{ord}, \kappa) \to \mathbb K(X)\),
    need not be an open mapping onto its range \(R = \mathrm{img}^{``}\mathbb K(X, \mathrm{ord}, \kappa)\).
\end{example}
\begin{proof}
    Let \(p = \langle 0,1 \rangle\) and \(X = \{ p \} \cup \{ \langle x,0 \rangle : x \in \mathbb R \}\) be viewed
    as a subspace of \(\mathbb R^2\).
    Identify each \(\langle x, 0 \rangle\) with \(x\).
    Consider \(\kappa = 4\) and, for \(\alpha \in \{0,1\}\), let \(V_\alpha = \{ p \}\).
    We show that \(\mathrm{img}^{``}[X; \{0,1\}, V]\) is not open relative to \(R\).
    First, note that \(\mathrm{img}^{``}[X; \{0,1\}, V]\) consists of subsets of \(X\) containing \(p\)
    of cardinality \(\leq 3\).
    Now, note that \(\{ 0, p \} \in \mathrm{img}^{``}[X; \{0,1\}, V]\).
    We claim that no open neighborhood of \(\{0,p\}\) relative to \(R\) is a subset of \(\mathrm{img}^{``}[X; \{0,1\}, V]\).
    Indeed, consider any basic open neighborhood \([U_1,\ldots, U_n]\) of \(\{0,p\}\) in \(\mathbb K(X)\).
    Choose \(j \in \{1,\ldots, n\}\) such that \(0 \in U_j\) and then choose
    \[\{x_1,x_2\} \in [U_j \cap \mathbb R \setminus \{0\}]^2.\]
    Conclusively,
    \[\{p,0,x_1,x_2\} \in [U_1, \ldots, U_n] \cap R \setminus \mathrm{img}^{``}[X; \{0,1\}, V],\]
    finishing the proof.
\end{proof}

We now summarize game-theoretic consequences of Propositions \ref{prop:ContinuousImage} and \ref{prop:ProperImage}.
\begin{corollary} \label{cor:FiniteInequality}
    For any space \(X\) and \(\ast \in \{1,\mathrm{fin}\}\),
    \[\mathsf G_\ast(\mathcal O_{\mathbb F(X,\mathrm{ord})},\mathcal O_{\mathbb F(X,\mathrm{ord})})
    \leq_{\mathrm{II}}^+ \mathsf G_\ast(\Omega_X,\Omega_X).\]
\end{corollary}
\begin{proof}
    Note that \(f \mapsto \mathrm{img}(f)\), \(\mathbb F(X,\mathrm{ord}) \to \mathcal P_{\mathrm{fin}}(X)\),
    is a continuous surjection by Proposition \ref{prop:ImageMappingQuotient}.
    Hence, by Proposition \ref{prop:ContinuousImage},
    \[\mathsf G_\ast(\mathcal O_{\mathbb F(X,\mathrm{ord})},\mathcal O_{\mathbb F(X,\mathrm{ord})})
    \leq_{\mathrm{II}}^+
    \mathsf G_\ast(\mathcal O_{\mathcal P_{\mathrm{fin}}(X)},\mathcal O_{\mathcal P_{\mathrm{fin}}(X)}).\]
    The desired conclusion obtains by Theorem \ref{thm:FiniteMotivation}.
\end{proof}
Note that \(\mathbb F(\omega, \mathrm{ord}) = \mathbb K(\omega, \mathrm{ord})\) is Menger since it is \(\sigma\)-compact
by Theorem \ref{thm:KOmegaOrd}.
Hence, there are nontrivial examples of spaces \(\mathbb F(X, \mathrm{ord})\) which are Menger.
Things are much less interesting in the single-selection context, however.
\begin{proposition} \label{prop:TrivialRothberger}
    For any \(T_1\) space \(X\), \(\mathbb F(X, \mathrm{ord})\) is Rothberger if and only if \(X\)
    is a singleton.
\end{proposition}
\begin{proof}
    If \(X\) is a singleton, then \(\mathbb F(X, \mathrm{ord})\) is also a singleton and clearly Rothberger.
    So suppose \(X\) is a \(T_1\) space with at least two distinct points \(p,q \in X\).
    Since \(X\) is \(T_1\), \(\{p,q\}\) is closed in \(X\) and is the discrete doubleton as a subspace of \(X\).
    Hence, by Proposition \ref{prop:BasicProductOfClosed}, \(Y := [\{p,q\}; \emptyset, \emptyset]\) is a closed subspace
    of \(\mathbb F(X, \mathrm{ord})\).
    Note that \(Y\) is homeomorphically equivalent to \(\mathsf V(2^\omega)\), which establishes that \(\mathbb F(X, \mathrm{ord})\)
    is not Rothberger, as argued in Example \ref{ex:CantorNonRothberger}.
\end{proof}
Consequently, we see that the single-selection version of Corollary \ref{cor:FiniteInequality} does not, in general, reverse.
\begin{example} \label{example:OmegaRothbergerNontransfer}
    In general, for a space \(X\),
    \[\mathsf G_1(\Omega_X, \Omega_X) \not\leq_{\mathrm{II}} \mathsf G_1(\mathcal O_{\mathbb F(X, \mathrm{ord})},\mathcal O_{\mathbb F(X, \mathrm{ord})}).\]
    Indeed, any \(T_1\) space \(X\) with at least two points which is \(\omega\)-Rothberger, like the discrete doubleton,
    has the property that \(\mathbb F(X, \mathrm{ord})\) is not Rothberger by Proposition \ref{prop:TrivialRothberger}.
\end{example}
Note however that spaces which are \(\sigma\)-compact are \(\omega\)-Menger (see \cite[Corollary 4.18]{TraditionalMenger}),
so the spaces in this paper explicitly shown to witness the condition of Example \ref{example:OmegaRothbergerNontransfer} do not extend
to the finite-selection context.
In fact, we have yet to identify any examples of a space extending Example \ref{example:OmegaRothbergerNontransfer} to finite-selections.
\begin{question} \label{question:MengerThing}
    Is it possible that, for spaces \(X\),
    \[\mathsf G_{\mathrm{fin}}(\Omega_X, \Omega_X)
    \leftrightarrows \mathsf G_{\mathrm{fin}}(\mathcal O_{\mathbb F(X, \mathrm{ord})},\mathcal O_{\mathbb F(X, \mathrm{ord})})?\]
\end{question}
We have not yet even been able to determine whether \(\mathbb F(\mathbb R, \mathrm{ord})\) is Menger;
if \(\mathbb F(\mathbb R, \mathrm{ord})\) fails to be Menger, then Question \ref{question:MengerThing} would be
answered in the negative.

In general, the mapping \(f \mapsto \mathrm{img}(f)\), \(\mathbb K(X, \mathrm{ord}) \to \mathbb K(X)\),
is not proper.
Indeed, consider \(\mathbb K(\omega,\mathrm{ord})\) and \(F := \{ K \in \mathbb K(\omega) : K \subseteq \{0,1\}\}\),
a compact subset of \(\mathbb K(\omega)\).
Note that \(\mathrm{img}^{-1}(F)\) is homeomorphically equivalent to \(\mathsf V(2^\omega)\),
and \(\mathsf V(2^\omega)\) is not compact as shown in Example \ref{example:Cantor}.

Nevertheless,
\begin{lemma} \label{lem:CompactCovering}
    The mapping \(f \mapsto \mathrm{img}(f)\), \(\mathbb K(X, \mathrm{ord}) \to \mathbb K(X)\),
    is a compact covering map.
\end{lemma}
\begin{proof}
    Since \(\mathbb K(X, \mathrm{ord}) = \mathbb K(X, \mathrm{ord}, \kappa)\), where
    \[\kappa = \sup \{ \#K : K \in K(X) \} + \omega \geq \omega,\]
    Proposition \ref{prop:ImageMappingQuotient} asserts that \(f \mapsto \mathrm{img}(f)\),
    \(\mathbb K(X, \mathrm{ord}) \to \mathbb K(X)\), is a continuous open surjection.
    Applying Proposition \ref{prop:OpenSurjection} finishes the proof.
\end{proof}
\begin{corollary} \label{cor:CompactInequality}
    For any space \(X\) and \(\ast \in \{ 1, \mathrm{fin} \}\),
    \[\mathsf G_\ast(\mathcal K_{\mathbb K(X,\mathrm{ord})},\mathcal K_{\mathbb K(X,\mathrm{ord})})
    \leq_{\mathrm{II}}^+ \mathsf G_\ast(\mathcal K_X, \mathcal K_X)\]
    and
    \[\mathsf G_\ast(\mathcal O_{\mathbb K(X,\mathrm{ord})},\mathcal O_{\mathbb K(X,\mathrm{ord})})
    \leq_{\mathrm{II}}^+ \mathsf G_\ast(\mathcal K_X, \mathcal K_X).\]
\end{corollary}
\begin{proof}
    By Lemma \ref{lem:CompactCovering}, \(f \mapsto \mathrm{img}(f)\), \(\mathbb K(X, \mathrm{ord}) \to \mathbb K(X)\),
    is a compact covering map.
    Consequently,
    \[\mathsf G_\ast(\mathcal K_{\mathbb K(X,\mathrm{ord})},\mathcal K_{\mathbb K(X,\mathrm{ord})})
    \leq_{\mathrm{II}}^+ \mathsf G_\ast(\mathcal K_{\mathbb K(X)}, \mathcal K_{\mathbb K(X)})\]
    and
    \[\mathsf G_\ast(\mathcal O_{\mathbb K(X,\mathrm{ord})},\mathcal O_{\mathbb K(X,\mathrm{ord})})
    \leq_{\mathrm{II}}^+ \mathsf G_\ast(\mathcal O_{\mathbb K(X)}, \mathcal O_{\mathbb K(X)}),\]
    by Propositions \ref{prop:ProperImage} and \ref{prop:ContinuousImage}, respectively.
    Since, by \cite[Corollary 4.14]{CHVietoris},
    \[\mathsf G_\ast(\mathcal O_{\mathbb K(X)}, \mathcal O_{\mathbb K(X)})
    \leq_{\mathrm{II}}^+ \mathsf G_\ast(\mathcal K_X, \mathcal K_X)
    \leftrightarrows \mathsf G_\ast(\mathcal K_{\mathbb K(X)}, \mathcal K_{\mathbb K(X)}),\]
    the asserted statements obtain.
\end{proof}
None of the inequalities appearing in Corollary \ref{cor:CompactInequality}, in general, reverse.
\begin{example} \label{example:ComactNontransfer}
    In general, for a space \(X\),
    \[\mathsf G_1(\mathcal K_X, \mathcal K_X) \not\leq_{\mathrm{II}} \mathsf G_1(\mathcal O_{\mathbb K(X, \mathrm{ord})},\mathcal O_{\mathbb K(X, \mathrm{ord})}),\]
    \[\mathsf G_1(\mathcal K_X, \mathcal K_X) \not\leq_{\mathrm{II}} \mathsf G_1(\mathcal K_{\mathbb K(X, \mathrm{ord})},\mathcal K_{\mathbb K(X, \mathrm{ord})}),\]
    \[\mathsf G_{\mathrm{fin}}(\mathcal K_X, \mathcal K_X) \not\leq_{\mathrm{II}} \mathsf G_{\mathrm{fin}}(\mathcal O_{\mathbb K(X, \mathrm{ord})},\mathcal O_{\mathbb K(X, \mathrm{ord})}),\]
    and
    \[\mathsf G_{\mathrm{fin}}(\mathcal K_X, \mathcal K_X) \not\leq_{\mathrm{II}} \mathsf G_{\mathrm{fin}}(\mathcal K_{\mathbb K(X, \mathrm{ord})},\mathcal K_{\mathbb K(X, \mathrm{ord})}).\]
\end{example}
\begin{proof}
    We show that \(\mathbb K(\mathbb R, \mathrm{ord})\) is not Lindel\"{o}f.
    Since \(\mathbb R\) is \(k\)-Rothberger, this suffices for the example (see \cite[Figures 1 and 2]{TraditionalMenger}).
    So consider
    \[\mathscr U := \{ (-1,1)^{\mathfrak c} \} \cup \left\{ \{ f \in \mathbb K(\mathbb R, \mathrm{ord}) : |f(\alpha)| > 1/2 \} : \alpha \in \mathfrak c \right\}.\]
    Note that \(\mathscr U\) is an open cover of \(\mathbb K(\mathbb R, \mathrm{ord})\).
    Then let \(\{ \alpha_n : n \in \omega \} \subseteq \mathfrak c\) be arbitrary and consider
    \[\mathscr V := \{ (-1,1)^{\mathfrak c} \} \cup \left\{ \{ f \in \mathbb K(\mathbb R, \mathrm{ord}) : |f(\alpha_n)| > 1/2 \} : n \in \omega \right\}.\]
    We can then choose \(\beta < \mathfrak c\) such that \(\sup \{ \alpha_n : n \in \omega \} < \beta\) and then
    consider the function \(f : \mathfrak c \to \mathbb R\) defined by
    \[f(\alpha) = \begin{cases} 0, & \alpha < \beta, \\ 2, & \beta \leq \alpha. \end{cases}\]
    Note that \(f \in \mathbb K(\mathbb R, \mathrm{ord}) \setminus\bigcup \mathscr V\), which establishes that \(\mathscr U\) has no countable subcover.
\end{proof}

Recall that an \emph{\(S\)-space} is a hereditarily separable space which is not Lindel\"{o}f.
As we have seen above, \(\mathbb K(\mathbb R, \mathrm{ord})\) is separable but not Lindel\"{o}f.
However, by a well-known argument\footnote{See, for example,\\ \url{https://dantopology.wordpress.com/2009/11/06/product-of-separable-spaces/}},
\(\mathbb K(\mathbb R, \mathrm{ord})\) is not an \(S\)-space.
\begin{proposition}
    The space \(\mathbb K(\mathbb R, \mathrm{ord})\) is not hereditarily separable.
\end{proposition}
\begin{proof}
    Consider
    \[X := \{ f \in 2^{\mathfrak c} : \# f^{-1}(1) \leq \aleph_0 \},\]
    a restricted subspace of the \(\Sigma\)-product \(\Sigma(\mathbb R,\mathbf 0,\mathfrak c)\).
    Note that \(X \subseteq \mathbb K(\mathbb R, \mathrm{ord})\).
    We show that \(X\) is not separable.
    Consider any countable subset \(\{ f_j : j \in \omega \} \subseteq X\)
    and choose \(\alpha \in \mathfrak c \setminus \bigcup \{ f^{-1}_j(1) : j \in \omega \}\).
    Then note that
    \[[\mathbb R; \{\alpha\} , \{ \langle \alpha , \mathbb R \setminus \{ 0 \} \rangle \}] \cap \{ f_j : j \in \omega \} = \emptyset,\]
    establishing the existence of a nonempty open subset of \(X\) which is disjoint from \(\{ f_j : j \in \omega \}\).
\end{proof}

\section*{Acknowledgments}
    The authors thank Jocelyn Bell for bringing the authors' attention to \cite{PinchedCube} and
    Bob Hingtgen for discussion relating to Example \ref{example:UniformBox} in an initial draft of the paper.

\end{document}